\documentclass{amsart}
\usepackage{amsfonts,graphicx,amssymb,stmaryrd,amscd,amsmath,latexsym,amsbsy,tikz,bm,multirow,array}
\usepackage{color}
\usepackage[all,cmtip]{xy}

\theoremstyle{plain}
\newtheorem{thm}{Theorem}[section]

\newtheorem{prop}[thm]{Proposition}

\newtheorem{lem}[thm]{Lemma}
\newtheorem{cor}[thm]{Corollary}

\theoremstyle{remark}
\newtheorem{remark}[thm]{Remark}

\numberwithin{equation}{section}

\newcommand{\ca}{\mathcal}
\newcommand{\f}{\mathfrak}
\newcommand{\n}{\textup}

\newcommand{\Z}{\mathbb{Z}}
\newcommand{\Q}{\mathbb{Q}}

\newcommand{\C}{\mathbb{C}}

\newcommand{\Id}{\n{Id}}
\newcommand{\End}{\n{End}}
\newcommand{\GL}{\n{GL}}
\newcommand{\Tr}{\mathop{\n{Tr}}}
\newcommand{\dom}{\n{dom}}
\newcommand{\Mer}{\n{Mer}}

\newcommand{\eqrefs}[2]{\textnormal{(\ref{#1}-\ref{#2})}}

\title[Boundary transfer matrices and boundary quantum KZ equations]{\hspace{9.3mm} Boundary transfer matrices and \newline boundary quantum KZ equations}
\author{Bart Vlaar}
\address{School of Mathematical Sciences, University of Nottingham, NG7 2RD, UK}
\email{Bart.Vlaar@nottingham.ac.uk}

\begin{document}
\begin{abstract}
A simple relation between inhomogeneous transfer matrices and boundary quantum KZ equations is exhibited for quantum integrable systems with reflecting boundary conditions, analogous to an observation by Gaudin for periodic systems.
Thus the boundary quantum KZ equations receive a new motivation.
We also derive the commutativity of Sklyanin's boundary transfer matrices by merely imposing appropriate reflection equations, in particular without using the conditions of crossing symmetry and unitarity of the R-matrix.
\end{abstract}

\maketitle

\section{Introduction}

Many interesting objects associated to integrable models and representation theory are known to satisfy difference or differential equations.
The Knizhnik-Zamolodchikov (KZ) equations are differential equations defining conformal blocks in Wess-Zumino-Witten conformal field theory \cite{KZ} and describing intertwiners between certain representations of affine Kac-Moody algebras \cite{TK}; they are also connected to algebraic number theory \cite{Drinfeld}.
Quantum deformations of these equations yield difference equations, known as quantum Knizhnik-Zamolodchikov (qKZ) equations; they arise independently as equations satisfied by correlation functions and form factors of quantum integrable models \cite{JM,Sm} and by matrix elements of intertwiners for representations of quantum affine algebras \cite{FR}.
More recently, connections with combinatorics have been extensively investigated when the (multiplicative) shift parameter in the equations assumes root-of-unity values, in conjunction with the Razumov-Stroganov conjectures for loop models, e.g.\ in \cite{DFZJ,GP,RSZJ,ZJ}.

Cherednik \cite{Ch1,Ch2} constructed generalizations of the qKZ equations in terms of a so-called R-matrix datum associated to arbitrary affine root systems.
In this framework the aforementioned ``original'' qKZ equations correspond to the case where the affine root system is of type A and are related to integrable systems with periodic boundary conditions (or rather twisted-periodic or quasi-periodic, but we will nevertheless use ``periodic'' in our terminology).
If the affine root system is of a different classical type (i.e. of type B, C or D) we arrive at the \emph{boundary qKZ equations} (bqKZ), relating to integrable systems with up to two reflecting boundaries.
In the current work we will focus on the case when the affine root system is of type C - since the affine Weyl groups of types B and D naturally appear as normal subgroups of the affine Weyl group of type C, many of our results can be straightforwardly modified to results for types B and D.

The bqKZ equations were first studied and motivated in their own right in \cite{JKKMW}, where they describe correlation functions of semi-infinite spin chains with integrable boundary conditions.
It would be interesting to find other motivations, analogous to the ones for the type A qKZ equations.
This paper provides such a motivation by showing that in the limit that the shift parameter goes to 1, the bqKZ equations are related, through Yang's notion of scattering matrices, to interpolants of commuting transfer matrices, which is entirely parallel to an observation by Gaudin \cite[Ch.\ 10]{Gaudin} for the type A case (also cf.\ \cite{Pasquier}).
It is beneficial to review some basic terminology in more detail at this stage.

\subsection{Quantum Knizhnik-Zamolodchikov equations and scattering matrices}

Consider a collection of complex vector spaces $V_1,\ldots,V_N$, called \emph{local state spaces}, and construct the \emph{global state space} $W = V_1 \otimes \cdots \otimes V_N$.
In this paper we will deal with the case $V_1 = \ldots = V_N =:V$ only\footnote{It is possible that this restriction may be lifted for some of the theory under consideration; we will return to this question in Subsection \ref{sec:outlook}.}.
For $p \in \C^\times := \C \setminus \{0\}$ and $i=1,\ldots,N$, write $p^{\epsilon_i}: (\C^\times)^N \to (\C^\times)^N$ for multiplication by $p$ in the $i$-th entry:
\[ p^{\epsilon_i} \bm z = (z_1,\ldots,z_{i-1},pz_i,z_{i+1},\ldots,z_N), \quad \n{for } \bm z = (z_1,\ldots,z_N) \in (\C^\times)^N. \]
The qKZ equations are the following system of $p$-difference equations for meromorphic functions $f: (\C^\times)^N \to V^{\otimes N}$:
\begin{equation} \label{eqn:qKZ} \quad f(p^{\epsilon_i} \bm z) = A_i(\bm z;p) f(\bm z), \qquad \n{for } i=1,\ldots,N, \end{equation}
for certain \emph{qKZ transport matrices} $A_i(\bm z;p) \in \End(V^{\otimes N})$ depending meromorphically on the $z_i$.
The system \eqref{eqn:qKZ} is consistent if the two possible ways of resolving $f(p^{\epsilon_i} p^{\epsilon_j} \bm z)$ amount to the same; this happens precisely if
\begin{equation} \label{eqn:qKZconsistency}
\quad A_i(p^{\epsilon_j} \bm z;p) A_j(\bm z;p) = A_j(p^{\epsilon_i} \bm z;p) A_i(\bm z;p), \qquad \n{for } i,j=1,\ldots,N.
\end{equation}
Thus, \eqref{eqn:qKZconsistency} can be viewed as a flatness condition for a discrete connection defined by \eqref{eqn:qKZ}.
We assume the $A_i$ are composed of ``local'' operators (i.e.\ acting on one or two tensorands $V$).
The main such ingredient is an R-matrix $R(x) \in \End(V \otimes V)$, i.e.\ a meromorphic solution of the (quantum) \emph{Yang-Baxter equation} (YBE):
\begin{equation} \label{YBE}
R_{12}(x/y) R_{13}(x) R_{23}(y) = R_{23}(y) R_{13}(x) R_{12}(x/y) \in \End(V^{\otimes 3})
\end{equation}
for generic $x,y$.
The subscripts indicate in which tensor factors the operator in question acts nontrivially, also cf.\ Appendix \ref{sec:linearalgebra}.
\eqref{YBE} can be seen as a condition expressing the equivalence of the two possible factorizations of a three-particle interaction into three two-particle interactions.
Then \eqref{eqn:qKZconsistency}, a global condition, follows from the YBE \eqref{YBE} and other local conditions.

In the type B, C and D cases the above R-matrix datum (solutions of the YBE) has to be extended by two operators $K^+(x), K^-(x) \in \End(V)$, meromorphically depending on $x$, the so-called K-matrices.
They are required to satisfy \emph{reflection equations} (REs), also known as boundary Yang-Baxter equations, one corresponding to either of two boundaries present in the system.
The \emph{left and right reflection equations} (LRE, RRE) are the relations
\begin{align}
\label{LRERK} R_{12}(x/y) K^+_1(x) R_{21}(xy) K^+_2(y) &= K^+_2(y) R_{12}(xy) K^+_1(x) R_{21}(x/y) , \\
\label{RRERK} R_{21}(x/y) K^-_1(x) R_{12}(xy) K^-_2(y) &= K^-_2(y) R_{21}(xy)  K^-_1(x) R_{12}(x/y),
\end{align}
both acting on $V^{\otimes 2}$;
here $R_{21}(x) = PR(x)P$ with the ``flip'' $P \in \GL(V^{\otimes 2})$ defined by $P(v \otimes v') = v' \otimes v$ for $v,v'\in V$.
\eqrefs{LRERK}{RRERK} express the equivalence of the two possible factorizations of a two-particle-and-boundary interaction into two one-particle-and-boundary interactions.
We will review the precise expression of the qKZ transport matrices in terms of the R- and K-matrices in Section \ref{sec:inhtfermatsandscatmats}.

In the special case that $p=1$ the consistency condition \eqref{eqn:qKZconsistency} simplifies to the statement that the matrices $A_i(\bm z;1)$ ($1\leq i \leq N$) mutually commute.
The $A_i(\bm z;1)$ correspond to the \emph{scattering matrices} introduced by Yang in his investigations into the delta Bose gas \cite{Yang}.
Scattering matrices predate the qKZ equations and play an important role in quantum integrability (cf.\ \cite[Ch.\ 10]{Gaudin} and \cite{Pasquier}).
Hence we may view \eqref{eqn:qKZconsistency} as a $p$-deformed criterion for integrability.
Moreover, the qKZ equations \eqref{eqn:qKZ} then naturally appear in this story as being deformations of the eigenvector equation for scattering matrices with eigenvalue 1, thus motivating the study of (solutions of) qKZ equations (also see e.g.\ the introductory remarks in \cite{ZJ}).

\subsection{Transfer matrices}
Another main criterion of quantum integrability is Baxter's notion of commuting \emph{transfer matrices} which is at the basis of the quantum inverse scattering method (algebraic Bethe ansatz) as developed by the Faddeev school from the 1980s; for textbook accounts see \cite{Ba,KBI}.
The transfer matrix, originally associated specifically to vertex models from statistical mechanics, is a parameter-dependent linear operator $T(x)$ acting on a state space $W$.
As before, if $W = V^{\otimes N}$, transfer matrices can be built up from local operators in such a way that the integrability criterion (commutavity) can be derived from local integrability conditions like the YBE.
Also, it is possible to introduce an additional dependence on an $N$-tuple of complex numbers (so-called inhomogeneities) into the transfer matrix, yielding inhomogeneous transfer matrices.
In Section \ref{sec:transfermatrices} we study transfer matrices in more detail.

The method of commuting transfer matrices is best understood in the case of systems with periodic or ``closed'' boundary conditions, but for systems with reflecting or ``open'' boundary conditions there is a more elaborate version due to Sklyanin \cite{Sk}, who constructed commuting \emph{boundary transfer matrices} from R- and K-matrices and derived the algebraic Bethe ansatz for special types of R- and K-matrices.
Similar to the setup for the qKZ transport matrices, two K-matrices are required.
One of these can be taken equal to one of the K-matrices featuring in the qKZ transport matrix, say $K^-$; thus it satisfies the RRE \eqref{RRERK}.
In the present context it is crucial that with that choice, the other necessary K-matrix $K'$ is in general not a solution of the LRE \eqref{LRERK}, but a third reflection equation \eqref{DRERK}.

\subsection{The connection between qKZ equations and transfer matrices}

Gaudin \cite[Ch.\ 10]{Gaudin} has highlighted that scattering matrices are proportional to interpolants of inhomogeneous transfer matrices (those with the spectral parameter running through the set of inhomogeneities), in case the underlying affine root system is of type A.
Thus the problem of finding eigenvectors of transfer matrices is related to the problem of finding eigenvectors of scattering matrices, adding to the relevance of the qKZ equations.
For the bqKZ equations, such a connection with inhomogeneous boundary transfer matrices has heretofore been unclear, as observed by Pasquier \cite{Pasquier}; this owes mainly to the fact that the K-matrices for the left boundary appearing in the formulae satisfy different reflection equations.
We address this question in Thm.\ \ref{thm:connection}, thus providing a new motivation for the bqKZ equations.
For the case where $V \cong \C^2$ and the R-matrix is of $U_q(\hat{\f{sl}}_2)$-type this connection was made in \cite{SV}; earlier it was also made in \cite{FW} for K-matrices diagonal in the Cartan basis of $V$.

\subsection{Outline}

Each of the main three sections of this paper is split up in a part about periodic systems and a part on reflecting systems, the former of which consists mainly of a review of existing results which is provided as a background for the new results proposed in the latter.

In Section \ref{sec:transfermatrices} we will discuss transfer matrices for periodic and reflecting systems for general state spaces $W$ (not necessarily tensor products of local state spaces $V_i$).
The new result here addresses an unsatisfactory aspect of the state-of-affairs of quantum integrability for reflecting systems, namely the large number of conditions on the R-matrix datum required to establish the commutativity of the boundary transfer matrices in \cite{Sk}, both compared to the analogon in the periodic case and to the requirements for the consistency of the bqKZ equations.
This problem has already been reduced significantly in \cite{MN} and \cite{FSHY} and in the current work the main improvement (cf.\  Thm.\ \ref{thm:commutingtransfermatrices}) is to show it is unnecessary to assume unitarity or crossing symmetry of the R-matrix, as in the periodic case, essentially leaving the appropriate REs at both boundaries as the only conditions.

From here onwards we will focus on the case $W = V^{\otimes N}$.
In Section \ref{sec:inhtfermatsandscatmats} we will discuss the connection between inhomogeneous transfer matrices and qKZ transport matrices.
For reflecting systems this leads to the main theorem Thm.\ \ref{thm:connection} of this paper, which settles the aforementioned problem identified by Pasquier: it establishes a simple relation between bqKZ transport matrices \cite{Ch1,Ch2,JKKMW} and the inhomogeneous boundary transfer matrices \cite{Sk,MN,FSHY}.
This relies on a careful analysis of the relation between the K-matrices $K^+$ and $K'$, q.v.\ Lemma \ref{lem:refleqns}.

As an application of these relations, in Section \ref{sec:tfermatsrevisited} we will derive the commutativity of transfer matrices from the qKZ consistency conditions for special classes of integrable systems, both periodic and reflecting, yielding a generalization of a result by Razumov, Stroganov and Zinn-Justin \cite{RSZJ}.

Finally, in Section \ref{sec:outlook} we will outline future work and possible generalizations.

\subsection{Some notational conventions}

For a finite-dimensional complex vector space $V$, we will saliently identify $V$ with its dual, so that transposition in $V$ becomes an algebra-antiautomorphism: $\End(V) \to \End(V): X \mapsto X^t$. 
For $X \in \GL(V)$ we write $X^{-t} := (X^{-1})^t = (X^t)^{-1}$.
Let $1 \leq i \leq N$ and $X \in \End(V_1 \otimes \cdots \otimes V_N)$ with $V_i$ finite-dimensional. 
We will write $X^{t_i}$ for the transpose of $X$ with respect to the $V_i$ and $\Tr_i X$ for the trace of $X$ with respect to $V_i$. 
We refer to Appendix \ref{sec:linearalgebra} for more detail on such relative or ``partial'' transposes and traces.\\

Let $\Mer$ denote the associative algebra of meromorphic functions$: \C \to \C$.
We write $\Mer^\times$ for the complement in $\Mer$ of the constant function zero.
Given a complex vector space $V$ (not necessarily finite-dimensional), let $\Mer(V) \cong \Mer \otimes \End(V)$ denote the associative algebra of meromorphic functions$: \C \to \End(V)$.
By $\Mer(V)^\times$ we denote the multiplicative subgroup of $\Mer(V)$ consisting of meromorphic functions ranging in $\End(V)$ whose images are generically invertible.

For $X,Y \in \End(V)$ the notation $X \propto Y$ denotes the equivalence relation $X = m Y$ for some $m \in \C^\times$.
If $X,Y \in \Mer(V)$ then $X(x) \propto Y(x)$ means $X(x) = m(x) Y(x)$ for generic values of $x$ and some $m \in \Mer^\times$.

\subsection{Acknowledgments}

The author was supported in his work by a Free Competition grant (``Double affine Hecke algebras, Integrable Models and Enumerative Combinatorics'') of the Netherlands Organization for Scientific Research (NWO) and by grant EP/L000865/1 (``Topological field theories, Baxter operators and the Langlands programme'') of the Engineering and Physical Science Research Council (EPSRC).
He would like to thank C. Korff, N. Reshetikhin, J. Stokman, R. Weston and P. Zinn-Justin for helpful comments and their interest in this work.

\section{Commuting transfer matrices} \label{sec:transfermatrices}

In this section we compare the requirements for the existence of boundary commuting transfer matrices with the periodic case and formulate an improved commutativity statement.
First we  summarize the method of commuting transfer matrices.

Let $W$ be a complex vector space, called \emph{state space}.
The transfer matrix $T$ is a distinguished element of $\Mer(W)$; the variable on which it depends as a meromorphic function is called the \emph{spectral parameter}.
The condition of integrability for systems modelled in this way is that transfer matrices with different spectral parameters should commute: $[T(x),T(y)]=0$.
The transfer matrix relates directly to the notion of partition function of statistical mechanical models; for quantum mechanical models the importance of the transfer matrix is that the \emph{quantum Hamiltonian} of the model can be expressed in terms of it, typically as the logarithmic derivative of the transfer matrix for a special value of the spectral parameter.

The transfer matrix $T$ can be constructed in terms of the partial trace with respect to a finite-dimensional linear space $V$ called \emph{auxiliary space}. 
The object is inside the trace, an element of $\End(V \otimes W)$, satisfies a quadratic relation (``exchange relation'') in $\End(V \otimes V \otimes W)$ involving an invertible solution of the Yang-Baxter equation.
From this quadratic relation one derives $[T(x),T(y)]=0$.

\subsection{Commuting transfer matrices in the periodic setting}

For periodic integrable systems one is given an operator $U \in \Mer(V \otimes W)$ known as the \emph{monodromy matrix}.
One may construct the transfer matrix $T \in \Mer(W)$ in terms of the partial trace with respect to $V$:
\begin{equation} \label{eqn:transfermatrixper} T_1(x)  := \Tr_0 U_{01}(x), \end{equation}
where we have labelled the auxiliary space $V$ by 0 and the state space $W$ by 1.
We have the following standard result.
\begin{thm}[E.g.\ \cite{Ba,KBI}] \label{thm:commutingtransfermatricesper}
Let $U \in \Mer(V \otimes W)$.
Suppose there exists $R \in \Mer(V^{\otimes 2})^\times$ such that the pair $(R,U)$ satisfies, for generic values of $x, y \in \C$, 
\begin{equation} R_{00'}(x/y)U_{01}(x)U_{0'1}(y) = U_{0'1}(y)U_{01}(x)R_{00'}(x/y) \in \End(V_{(0)} \otimes V_{(0')} \otimes W_{(1)}). \label{YBERU} \end{equation}
Then the transfer matrices defined by \eqref{eqn:transfermatrixper} form a commuting family:
\[ \hspace{40mm} [T(x),T(y)]=0, \qquad \n{for all } x,y \in \dom(T). \]
\end{thm}

Given an R-matrix $R$, it can be easily checked that compositions of solutions of \eqref{YBERU} acting in different state spaces are again solutions of \eqref{YBERU} with the state space given by the tensor product of the first two state spaces.
Hence we have (see e.g. \cite[Sec.\ 2]{Sk})
\begin{cor} \label{cor:commutingtransfermatricesper2}
Let $W^+, W^-$ be complex vector spaces and let $V$ be a finite-dimensional vector space.
Let $U^\pm \in \Mer(V \otimes W^\pm)$.
Suppose there exists $R \in \Mer(V^{\otimes 2})^\times$ such that $(R,U^\pm)$ satisfies \eqref{YBERU} in $V \otimes V \otimes W^\pm$ for both choices of sign.
Then the transfer matrices $T \in \Mer(W^+ \otimes W^-)$ defined by
\[ T_{12}(x) := \Tr_0 U^+_{01}(x) U^-_{02}(y) \]
form a commuting family:
\[ \hspace{40mm} [T(x),T(y)]=0, \qquad \n{for all } x,y \in \dom(T). \]
\end{cor}

\begin{remark} \label{rem:D}
Typically, in Cor. \ref{cor:commutingtransfermatricesper2} one sets $W^+ \cong \C$ and takes $U^+$ to be a constant $D \in \GL(V)$ such that $(R,D)$ satisfies \eqref{YBERU} in $\End(V \otimes V)$, viz.
\begin{equation} [R(x),D\otimes D] = 0 \in \End(V \otimes V) \label{RDD}. \end{equation}
Then $T$ as given in Cor. \ref{cor:commutingtransfermatricesper2}, i.e.
\[ T_1(x) = \Tr_0 D_0 U^-_{01}(x) \]
can be viewed as the generating function of the integrals of motion of a quantum integrable system with state space $W \cong W^-$.
In this picture $U^-$ represents the contributions of the (particle-particle) interactions in the bulk and $D$ encodes the twist in the periodic boundary conditions, with $D = \Id_V$ pertaining to the special case of untwisted periodic boundary conditions.
\end{remark}

\subsection{Commuting boundary transfer matrices} \label{subsec:boundarytransfermatrices}

The notion of commutative transfer matrices in the reflecting setting was first investigated by Sklyanin in the seminal paper \cite{Sk}.
It involves the \emph{left and right boundary monodromy matrices}\footnote{Where an object in the periodic case has a direct analogon in the reflecting case, the symbol representing the latter will be the same as the former, but set in a calligraphic typeface.} $\ca U^+ \in \Mer(V \otimes W^+)$ and $\ca U^- \in \Mer(V \otimes W^-)$, respectively.
The \emph{boundary transfer matrix} is a generating function $\ca T \in \Mer(W^+ \otimes W^-)$ for the constants of motion of a quantum integrable system with state space $W^+ \otimes W^-$.
Analogous to the formula in Cor. \ref{cor:commutingtransfermatricesper2}, it is given by
\begin{equation} \label{defn:transfermatrixbdy}
\ca T_{12}(x) = \Tr_0 \ca U^+_{01}(x) \ca U^-_{02}(x), \end{equation}
where the product of operators inside the partial trace acts in $V_{(0)} \otimes W^+_{(1)} \otimes W^-_{(2)}$.
In general, for the reflecting case the two factors $\ca U^\pm$ cannot be merged into a single operator satisfying a simple relation involving $R$.

Conditions on $\ca U^\pm$ need to be imposed (cf.\ \cite{Sk,MN,FSHY}) to derive the commutativity of the $\ca T$.
As in the periodic setting these conditions involve an R-matrix $R \in \Mer(V^{\otimes 2})^\times$.
The RRE for the pair $(R,\ca U^-)$ is the relation
\begin{equation} \begin{aligned}
&R_{0'0}(x/y) \ca U^-_{01}(x) R_{00'}(xy) \ca U^-_{0'1}(y) = \\
& \qquad = \ca U^-_{0'1}(y) R_{0'0}(xy) \ca U^-_{01}(x) R_{00'}(x/y) \end{aligned} \qquad \in \End(V_{(0)} \otimes V_{(0')} \otimes W^-_{(1)}
\label{RRERU} \end{equation}
for generic values of $x, y \in \C$,
which is a direct generalization of the RRE \eqref{RRERK} that plays a role in the consistency of the boundary qKZ equations.
The condition we impose on $\ca U^+$ will, however, not be such a generalization of the LRE \eqref{LRERK}.
Assume that $R(x)^{t_1} \in \End(V_{(1)} \otimes V_{(2)})^\times$ for generic values of $x$, which condition we will abbreviate by $R^{t_1} \in \Mer(V^{\otimes 2})^\times$; this rules out the constant solution $R(x) = P$ of \eqref{YBE}.
Now introduce
\begin{equation} \label{eqn:Rtilde} \tilde R(x) := ((R(x)^{t_1})^{-1})^{t_1} . \end{equation}
The \emph{dual reflection equation} (DRE) for the pair $(R,\ca U^+)$ is the relation
\begin{equation} \label{DRERU}
\begin{aligned}
&R_{0'0}(x/y)^{-t}  \ca U^+_{01}(x)^{t_0} \tilde R_{00'}(x y)^t \ca U^+_{0'1}(y)^{t_{0'}} =  \\
& \qquad =\ca U^+_{0'1}(y)^{t_{0'}} \tilde R_{0'0}(x y)^t \ca U^+_{01}(x)^{t_0} R_{00'}(x/y)^{-t}
\end{aligned}\hspace{3mm} \in \End(V_{(0)} \otimes V_{(0')} \otimes W^+_{(1)})
\end{equation}
for generic values of $x, y \in \C$. \\

In the existing derivations of the commutativity of the boundary transfer matrices $[\ca T(x), \ca T(y)]=0$ further conditions have been used.
\begin{itemize}
\item Sklyanin \cite{Sk} imposed the conditions $R_{21}(x) = R(x)$ (P-symmetry) and $R(x)^t = R(x)$ (T-symmetry).
These conditions were replaced by the single condition $R_{21}(x) = R(x)^t$ (PT-symmetry) at the hands of Mezincescu and Nepomechie \cite{MN}.
Fan, Shi, Hou and Yang \cite{FSHY} presented a modification of Sklyanin's argument which did not rely on any such conditions.
\item We emphasize that in all of \cite{Sk,MN,FSHY} the following condition, called \emph{unitarity}, is assumed:
\begin{equation}
R(x) R_{21}(x^{-1}) \propto \Id_{V^{\otimes 2}}. \label{Runitary}
\end{equation}
\item \emph{Crossing symmetry} is the condition that there exists $M \in \GL(V)$ and $r \in \C^\times$ such that
\begin{equation} \label{CS}  M_2^{-1} R_{12}(r^2x)^{-1} M_2 \propto \tilde R_{12}(x) \quad \n{for generic } x, \end{equation}
i.e.
\[ (((R_{12}(r^2x)^{-1})^{t_1})^{-1})^{t_1} \propto M_2 R_{12}(x) M_2^{-1}. \]
One usually also imposes the following compatibility condition:
\begin{equation} [R(x),M \otimes M]=0 \in \End(V \otimes V). \label{RMM} \end{equation}
In \cite{Sk,MN,FSHY} the condition \emph{crossing unitarity}, the result of combining \eqref{Runitary} and \eqref{CS}, is assumed to derive the commutativity of the boundary transfer matrices.
In \cite{Sk, FSHY} only $M \propto \Id_V$ occurs, so that \eqref{RMM} is trivially true.
\end{itemize}

Conditions \eqrefs{Runitary}{RMM} are not as stringent as P- or T-symmetry: they appear naturally in the context of the (finite-dimensional) representation theory of quantum affine algebras $U_q(\hat{\f g})$, see e.g.\ \cite{EFK,FR}.
In particular, if $R$ is an intertwiner of a tensor product of finite-dimensional $U_q(\hat{\f g})$-modules, then we have \eqrefs{CS}{RMM} with $M=\n{diag}(q^{2\rho})$ and $r=q^{h^\vee}$, with $\rho$ the half-sum of positive roots of $\f g$ and $h^\vee$ the dual Coxeter number of $\f g$.
If the $U_q(\hat{\f g})$-modules are also irreducible, \eqref{Runitary} holds.

It would nevertheless be pleasing theoretically if the commutavity of the $\ca T(x)$ can be derived without \eqrefs{Runitary}{RMM}, as can be done in the periodic case, cf.\ Thm.\ \ref{thm:commutingtransfermatricesper}.
This is possible by a natural generalization of the proofs given in \cite{Sk,FSHY}.
\begin{thm} \label{thm:commutingtransfermatrices}
Suppose we have $\ca U^\pm \in \Mer(V \otimes W^\pm)$.
If there exists $R \in \Mer(V^{\otimes 2})^\times$ with $R^{t_1} \in \Mer(V^{\otimes 2})^\times$ such that \eqref{RRERU} and \eqref{DRERU} are satisfied, then the boundary transfer matrices $\ca T(x)$ defined by \eqref{defn:transfermatrixbdy} form a commuting family of operators:
\[ \hspace{40mm} [\ca T(x),\ca T(y)]=0, \qquad \n{for all } x,y \in \dom(T). \]
\end{thm}
\begin{proof}
The following argument holds for generic values of $x,y$.
We have $\ca T(x)_{12} = \Tr_0 \ca U^+_{01}(x)^t \ca U^-_{02}(x)^{t_0}$ due to \eqref{traceoftransposes}.
We have
\begin{align*}
 \ca T_{12}(x)\ca T_{12}(y) &= \Bigl( \Tr_0 \ca U^+_{01}(x)^{t_0} \ca U^-_{02}(x)^{t_0} \Bigr) \Bigl( \Tr_{0'} \ca U^+_{0'1}(y) \ca U^-_{0'2}(y)  \Bigr) \\
&= \Tr_{0,0'} \ca U^+_{01}(x)^{t_0} \ca U^+_{0'1}(y)  \ca U^-_{02}(x)^{t_0}  \ca U^-_{0'2}(y)  \displaybreak[2] \\
&= \Tr_{0,0'} \ca U^+_{01}(x)^{t_0} \ca U^+_{0'1}(y)  \tilde R_{00'}(xy)^{t_0} R_{00'}(xy)^{t_0} \ca U^-_{02}(x)^{t_0}  \ca U^-_{0'2}(y)  \displaybreak[2] \\
&= \Tr_{0,0'} \bigl( \ca U^+_{01}(x)^{t_0}  \tilde R_{00'}(xy)^t \ca U^+_{0'1}(y)^{t_{0'}} \bigr)^{t_{0'}} \bigl( \ca U^-_{02}(x) R_{00'}(xy)  \ca U^-_{0'2}(y) \bigr)^{t_0}.
\end{align*}
by subsequently applying \eqref{productoftraces}, inserting the identity $\tilde R_{00'}(xy)^{t_0} R_{00'}(xy)^{t_0} = \Id_{V^{\otimes 2}}$ and applying \eqref{transposeofproduct}.
Hence, by \eqref{traceoftransposes},
\begin{equation} \label{eqn:intermediate}
\ca T_{12}(x) \ca T_{12}(y) = \Tr_{0,0'} \bigl( \ca U^+_{01}(x)^{t_0}  \tilde R_{00'}(xy)^t \ca U^+_{0'1}(y)^{t_{0'}} \bigr)^{t_{0,0'}} \ca U^-_{02}(x) R_{00'}(xy) \ca U^-_{0'2}(y),
\end{equation}
where we have written $X^{t_{0,0'}} = (X^{t_0})^{t_{0'}}$.
Inserting the identity $R_{0'0}(\tfrac{x}{y})^{-1} R_{0'0}(\tfrac{x}{y}) = \Id_{V^{\otimes 2}}$ and using \eqref{transposeofproduct} again, we have
\begin{align*}
\ca T_{12}(x) \ca T_{12}(y) &= \Tr_{0,0'} \bigl( R_{0'0}(\tfrac{x}{y})^{-t} \ca U^+_{01}(x)^{t_0}  \tilde R_{00'}(xy)^t \ca U^+_{0'1}(y)^{t_{0'}} \bigr)^{t_{0,0'}} \cdot \\
& \hspace{30mm} \cdot R_{0'0}(\tfrac{x}{y})  \ca U^-_{02}(x) R_{00'}(xy) \ca U^-_{0'2}(y).
\end{align*}
Applying the reflection equations \eqref{RRERU} and \eqref{DRERU}, as well as \eqrefs{transposeofproduct} {commuteinsidetrace}, we obtain
\begin{align*}
\ca T_{12}(x) \ca T_{12}(y) &= \Tr_{0,0'}  \bigl( \ca U^+_{0'1}(y)^{t_{0'}} \tilde R_{0'0}(xy)^t \ca U^+_{01}(x)^{t_0}  R_{00'}(\tfrac{x}{y})^{-t}  \bigr)^{t_{0,0'}} \cdot \\
& \hspace{30mm} \cdot \ca U^-_{0'2}(y)R_{0'0}(xy) \ca U^-_{02}(x) R_{00'}(\tfrac{x}{y}) \\
&= \Tr_{0,0'} \bigl( \ca U^+_{0'1}(y)^{t_{0'}} \tilde R_{0'0}(xy)^t \ca U^+_{01}(x)^{t_0}   \bigr)^{t_{0,0'}} \ca U^-_{0'2}(y)R_{0'0}(xy) \ca U^-_{02}(x). \end{align*}
This equals $\ca T_{12}(y) \ca T_{12}(x)$ by virtue of \eqref{eqn:intermediate}.
Because the meromorphic function $[\ca T(x),\ca T(y)]$ is zero for generic $x,y$, it must be zero for all $x,y \in \dom(\ca T)$.
\end{proof}

\begin{remark}
Thm.\ \ref{thm:commutingtransfermatrices} shows that the conditions required for commutativity of transfer matrices for integrable models with state space $W^+ \otimes W^-$ are no more stringent in the reflecting case than the periodic, cf.\ Cor.\ \ref{cor:commutingtransfermatricesper2}.
In both cases the transfer matrix is the partial trace over an auxiliary space $V$ of a composition of two operators, each of which is required to satisfy a quadratic relation involving an object $R \in \Mer(V \otimes V)^\times$.
No further assumptions need to be made, apart from the very weak condition $R^{t_1} \in \Mer(V \otimes V)^\times$ in the reflecting case.
\end{remark}

\subsubsection{Double-row boundary monodromy matrices}
The following proposition gives a canonical way of constructing boundary monodromy matrices $\ca U^\pm \in \Mer(V \otimes W^\pm)$ from ``ordinary'' monodromy matrices $U^\pm \in \Mer(V \otimes W^\pm)$ (global solutions to the YBE) and K-matrices $K', K^- \in \Mer(V)$ (local solutions to the appropriate REs).
\begin{prop}[{cf. \cite[Prop.\ 2 and 3]{Sk}}] \label{folding}
Let $K', K^- \in \Mer(V)$ and let $U^\pm \in \Mer(V \otimes W^\pm)^\times$.
Define $\ca U^\pm \in \Mer(V_{(0)} \otimes W^\pm_{(1)})$ by means of
\begin{align*}
\ca U^+_{01}(x) &:= \bigl( \bigl(U^+_{01}(x^{-1})^{-1}\bigr)^{t_0} K'_0(x)^t  U^+_{01}(x)^{t_0} \bigr)^{t_0} \\
\ca U^-_{01}(x) &:=  U^-_{01}(x^{-1})^{-1} K^-_0(x) U^-_{01}(x)
\end{align*}
and write $U \in \Mer(V \otimes W^+ \otimes W^-)^\times$ for the operator defined by $U_{012}(x)=U^-_{02}(x)U^+_{01}(x)$.
Writing $W = W^+ \otimes W^-$, the boundary transfer matrices $\ca T \in \Mer(W)$ defined by \eqref{defn:transfermatrixbdy} satisfy
\begin{equation} \label{eqn:foldedtransfermatrix} \ca T_1(x) = \Tr_0 K'_0(x) U_{01}(x^{-1})^{-1} K^-_0(x)U_{01}(x), \end{equation}
where the operator inside the partial trace acts in $V_{(0)} \otimes W_{(1)}$.

Furthermore, if there exists $R \in \Mer(V^{\otimes 2})^\times$ with $R_{12}^{t_1} \in \Mer(V^{\otimes 2})^\times$ such that $(R,K')$ satisfies the DRE in $V \otimes V$, viz.
\begin{equation} \label{DRERK}
R_{12}(x/y)^{-1}  K'_1(x) \tilde R_{21}(x y)K'_2(y) = K'_2(y) \tilde R_{12}(x y) K'_1(x)  R_{21}(x/y)^{-t},  \end{equation}
$(R,K^-)$ satisfies the RRE \eqref{RRERK} and $(R,U)$ satisfies the YBE \eqref{YBERU},
then the $\ca T$ form a commuting family: $[\ca T(x),\ca T(y)]=0$.
\end{prop}

\begin{proof}
Using \eqref{transposeofproduct} and \eqref{Pintrace} we may re-write $\ca U^+$ as follows
\begin{align*}
\ca U^+_{01}(x) &=  \bigl( \bigl(K'_0(x) U^+_{01}(x^{-1})^{-1}\bigr)^{t_0}  U^+_{01}(x)^{t_0} \bigr)^{t_0} \\
&= \Tr_{0'} P_{00'} K'_0(x) U^+_{01}(x^{-1})^{-1}  U^+_{0'1}(x) \quad = \Tr_{0'} K'_{0'}(x) U^+_{0'1}(x^{-1})^{-1} P_{00'} U^+_{0'1}(x).
\end{align*}
Hence, applying \eqref{productoftraces} we have
\begin{align*}
\ca T_{12}(x) = \Tr_0 \ca U^+_{01}(x) \ca U^-_{02}(x) &= \Tr_{0,0'} K'_{0'}(x) U^+_{0'1}(x^{-1})^{-1} P_{00'} \ca U^-_{02}(x) U^+_{0'1}(x) \\
&= \Tr_{0'} K'_{0'}(x) U^+_{0'1}(x^{-1})^{-1} \bigl( \Tr_0 P_{00'} \ca U^-_{02}(x) \bigr) U^+_{0'1}(x) \\
&= \Tr_{0'} K'_{0'}(x) U^+_{0'1}(x^{-1})^{-1} \ca U^-_{0'2}(x) U^+_{0'1}(x).
\end{align*}
Using the expression for $\ca U^-$ in terms of $U^-$ and $K^-$ we obtain \eqref{eqn:foldedtransfermatrix}.

Note that \eqref{DRERK} is obtained from \eqref{DRERU} by setting $W^+ \cong \C$ and transposing the equation.
Furthermore \eqref{RRERU} is a straightforward consequence of \eqref{RRERK} and \eqref{YBERU}.
Hence, the statement of the commutativity of the $\ca T$ is now a consequence of Thm. \ref{thm:commutingtransfermatrices} with $W^+$ set to $\C$ and $W^-$ set to $W$.
\end{proof}

\begin{remark} \label{rem:Kprime}
We see that in the reflecting case it is also possible to choose $W^+ \cong \C$; unlike in the periodic case the associated solution $K'$ of the relation in $V \otimes V \otimes W^+$ should not be assumed to be constant.
The transfer matrix $\ca T$ is a generating function of integrals of motion of a quantum system interacting with two boundaries; the remaining boundary monodromy matrix $\ca U^-$ represents the interactions in the bulk and at the right boundary, and $K'$ is associated to the interactions at the left boundary.
This choice of separating one of the boundaries in this way matches the standard embedding of the finite Dynkin diagram of type C into the corresponding affine diagram, see Fig. \ref{fig:Dynkin}.
A similar parallel may be drawn in the periodic case, cf. Rem. \ref{rem:D}, with the finite and affine Dynkin diagrams of type A.
\begin{figure}[h]
\caption{Affine Dynkin diagrams for type A and C, respectively. The associated finite diagrams are obtained by removing the black dots and the incident edges.}
\label{fig:Dynkin}
\includegraphics[width=150mm,clip=true,trim=32mm 240mm 0 40mm]{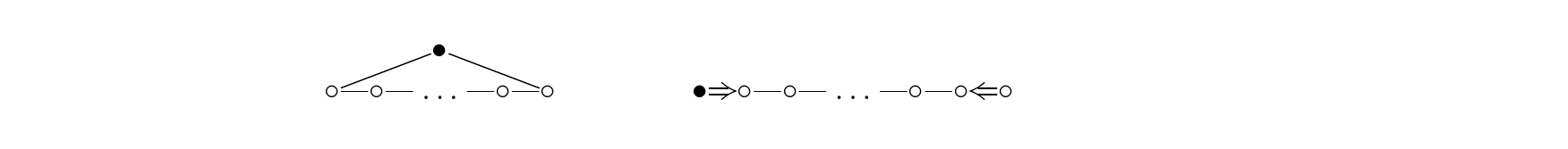}
\end{figure}
\end{remark}

\emph{Boundary regularity} for $K \in \Mer(V)$ is the ``initial condition''
\begin{equation} K( \pm 1) \propto \Id_V. \label{Kregular}
\end{equation}
From Prop. \eqref{folding} we immediately derive
\begin{lem} \label{lem:transfermatrixatone}
Let $K', K^- \in \Mer(V)$ and $U \in \Mer(V \otimes W)$, and let $\ca T$ be given by \eqref{eqn:foldedtransfermatrix}.
If $K^-$ satisfies \eqref{Kregular} and if $\pm 1 \in \dom(K')$, then
\[ \ca{T}(\pm 1) \propto \bigl( \Tr K'(\pm 1) \bigr) \Id_W. \]
\end{lem}

\section{Inhomogeneous transfer matrices and scattering matrices} \label{sec:inhtfermatsandscatmats}
We will now review Gaudin's observation that the operator appearing in the qKZ equations of type A and inhomogeneous transfer matrices for periodic systems are related and prove our main theorem, which is the analogon of this for the reflecting case.
The various statements are most naturally expressed in the language of Heisenberg spin chains.
Integrable inhomogeneous Heisenberg spin chains are obtained when the state space $W$ is a tensor product of ``local'' state spaces $V_i$, each connected to the states of one the spins in the chain.
We will also associate to each tensor factor a parameter $z_i \in \C^\times$ ($1 \leq i \leq N$), called an \emph{inhomogeneity}.

\subsection{Periodic systems}
As in the introduction, we will focus on the special case $W = V^{\otimes N}$ for some positive integer $N$, and assume the monodromy matrix $U$ factorizes into a product of R-matrices.
This yields
\begin{prop}[E.g.\ \cite{Ba,KBI}] \label{spinchain}
Let $R \in \Mer(V^{\otimes 2})$ and $\bm z = (z_1,\ldots,z_N) \in (\C^\times)^N$.
Suppose $R$ satisfies the YBE \eqref{YBE}.
Define $U \in \Mer(V \otimes V^{\otimes N})$ by
\[ U_{0}(x) = U_{0}(x;\bm z) :=R_{0N}(\tfrac{x}{z_N}) \cdots R_{01}(\tfrac{x}{z_1}), \]
where we have labelled the first tensor factor $V$ by 0 and suppressed the labels $1,\ldots,N$ of the remaining tensor factors.
Then \eqref{YBERU} holds.
\end{prop}

Starting from the datum $(R,D,\bm z)$, where $R \in \Mer(V^{\otimes 2})^\times$, $D \in \End(V)$ and $\bm z \in (\C^\times)^N$, assume \eqref{YBE} and the compatibility condition \eqref{RDD}.
For $W=V^{\otimes N}$, with $U$ given as in Prop.\ \ref{spinchain}, the transfer matrix of an inhomogeneous periodic spin chain is given by, cf.\ \eqref{eqn:transfermatrixper},
\begin{equation} \label{eqn:inhomogeneoustransfermatrixper} T(x) = T(x;\bm z) = \Tr_0 D_0 R_{0N}(\tfrac{x}{z_N}) \cdots R_{01}(\tfrac{x}{z_1}). \end{equation}
From Theorem \ref{thm:commutingtransfermatricesper} and Proposition \ref{spinchain} we derive
\begin{cor} \label{cor:commutinginhtransfermatricesper}
Let $R \in \Mer(V^{\otimes 2})^\times$, $D \in \End(V)$ and $\bm z = (z_1,\ldots,z_N) \in (\C^\times)^N$.
Construct the inhomogeneous transfer matrices according to \eqref{eqn:inhomogeneoustransfermatrixper}.
Suppose the YBE \eqref{YBE} and the compatibility condition \eqref{RDD} are satisfied.
Then we have
\[\qquad [ T(x;\bm z), T(y;\bm z)] = 0 \qquad \n{for generic } x,y,\bm z.\]
\end{cor}

To complete the description of the qKZ equations given in the Introduction for the type A case we define the qKZ transport matrices $A_i(\bm z;p)$ in terms of the R-matrices $R \in \Mer(V^{\otimes 2})^\times$ and the constant linear operator $D\in \GL(V)$ as follows:
\begin{equation} \label{qKZtransportmatrixdefnper}
A_i(\bm z;p) := R_{i \, i-1}(\tfrac{p z_i}{z_{i-1}}) \cdots R_{i1}(\tfrac{p z_i}{z_1}) D_i R_{Ni}(\tfrac{z_N}{z_i})^{-1} \cdots R_{ i+1 \, i}(\tfrac{z_{i+1}}{z_i})^{-1} . \end{equation}
For these operators we have (see, e.g.\ \cite[Thm.\ 5.4]{FR}) the following statement.
\begin{prop} \label{prop:qKZconsistencyper}
Suppose we have $R \in \Mer(V^{\otimes 2})^\times$ and $D \in \End(V)$ satisfying \eqref{YBE} and \eqref{RDD} and let $A_i(\bm z;p)$ be given by \eqref{qKZtransportmatrixdefnper}.
Then the consistency conditions \eqref{eqn:qKZconsistency} hold true.
\end{prop}

\begin{remark}
The local conditions imposed in Cor.\ \ref{cor:commutinginhtransfermatricesper} and Prop.\ \ref{prop:qKZconsistencyper} are manifestly the same.
As we will see in Subsection \ref{sec:inhtfermatsandscatmatsreflecting}, for reflecting integrable systems the situation is more subtle.
\end{remark}

\emph{Regularity} for $R \in \Mer(V^{\otimes 2})$ is the condition
\begin{equation}
R(1) \propto P \label{Rregular};
\end{equation}
note that if \eqref{Rregular} holds true, then by setting $x=1$ in the YBE \eqref{YBE} one obtains the unitarity condition \eqref{Runitary}.

\begin{remark}
The condition \eqref{Rregular} is required to express the Hamiltonian of the homogeneous XXZ spin chain in terms of $\frac{\mathrm d}{\mathrm d x} \log T(x;(1,\ldots,1))|_{x=1}$; this relation is due to Lieb \cite{Lieb}.
\end{remark}

The following result involving Yang's scattering matrices appears to have been observed for the first time by Gaudin \cite[Ch.\ 10]{Gaudin}.
\begin{thm} \label{thm:connectionper}
Assume that $R$ satisfies \eqref{Runitary} and \eqref{Rregular}.
Then the interpolants of the transfer matrix satisfy
\[ \hspace{36mm}  T(z_i;\bm z)  \propto A_i(\bm z;1) \qquad \n{for } i=1,\ldots,N, \]
for generic values of $\bm z$ and hence $[T(x;\bm z),T(y;\bm z)]=0$ for $x,y \in \{z_1,\ldots,z_N\}$.
\end{thm}

\begin{proof}
Using \eqref{Runitary} and \eqref{Rregular} the statement is easily obtained by straightforwardly moving the permutation operator $P$ past various $R$-matrices:
\begin{align*}
\hspace{-2mm} T(z_i;\bm z) &\propto \Tr_0 D_0 R_{0N}(\tfrac{z_i}{z_N}) \cdots R_{0 \, i+1}(\tfrac{z_i}{z_{i+1}}) P_{0i}R_{0 \, i-1}(\tfrac{z_i}{z_{i-1}}) \cdots R_{0 1}(\tfrac{z_i}{z_1}) \displaybreak[2] \\
&= \Tr_0 P_{0i} D_i R_{iN}(\tfrac{z_i}{z_N}) \cdots R_{i \, i+1}(\tfrac{z_i}{z_{i+1}}) R_{0 \, i-1}(\tfrac{z_i}{z_{i-1}}) \cdots R_{0 1}(\tfrac{z_i}{z_1}) \displaybreak[2] \\
&= \Bigl( \Tr_0 P_{0i}R_{0 \, i-1}(\tfrac{z_i}{z_{i-1}}) \cdots R_{0 1}(\tfrac{z_i}{z_1}) \Bigr) D_i R_{iN}(\tfrac{z_i}{z_N}) \cdots R_{i \, i+1}(\tfrac{z_i}{z_{i+1}}) \displaybreak[2] \\
&= R_{i \, i-1}(\tfrac{z_i}{z_{i-1}}) \cdots R_{i 1}(\tfrac{z_i}{z_1}) \Bigl( \Tr_0 P_{0i}\Bigr) D_i  R_{iN}(\tfrac{z_i}{z_N}) \cdots R_{i \, i+1}(\tfrac{z_i}{z_{i+1}}) \; \propto A_i(\bm z;1). \qedhere
\end{align*}
\end{proof}

\begin{remark}
For periodic systems, the unitarity condition \eqref{Runitary} is not necessary to derive the qKZ consistency or the commutativity of the transfer matrices, cf.\ Corollaries \ref{cor:commutinginhtransfermatricesper} and \ref{prop:qKZconsistencyper}.
However, because of the inverted appearance of certain R-matrices in $A_i(\bm z;p)$ cf.\ \eqref{qKZtransportmatrixdefnper}, it \emph{is} required for Thm.\ \ref{thm:connectionper}.
Alternatively, we could have chosen to replace the inverted R-matrices $R_{ji}(\tfrac{z_j}{z_i})^{-1}$ in the definition of $A_i(\bm z;p)$ by their ``unitarity-counterparts'' $R_{ij}(\tfrac{z_i}{z_j})$.
With this choice, \eqref{Runitary} would be necessary for the consistency condition of the $A_i$ instead of the proof of $T(z_i;\bm z) \propto A_i(\bm z;1)$.
We argue that with a view of extending Thm.\ \ref{thm:connectionper} to the reflecting case it is more natural to use our choice of $A_i(\bm z;p)$: it will turn out that for the proof of the version of this theorem for reflecting systems both the YBE and regularity are needed (so that unitarity is guaranteed, as well).
\end{remark}

\subsection{Reflecting systems} \label{sec:inhtfermatsandscatmatsreflecting}
Whereas the connection between transfer matrix and qKZ transport matrix has long been understood in the periodic case, in the reflecting case it has been less clear and it is the main purpose of this paper to show that this very connection holds true.
Given the datum $(R,K',K^-,\bm z)$, with $R \in \Mer(V^{\otimes 2})^\times$, $K',K^- \in \Mer(V)$ and $\bm z \in (\C^\times)^N$ such that \eqref{YBE}, \eqref{RRERK} and \eqref{DRERK} are satisfied, cf.\ Prop.\ \ref{folding} the corresponding boundary monodromy matrix $\ca U \in \Mer(V \otimes V^{\otimes N})$ satisfying \eqref{RRERU} is given by
\[ \ca U_0(x;\bm z)  :=  R_{01}(\tfrac{1}{xz_1})^{-1} \cdots R_{0N}(\tfrac{1}{xz_N})^{-1} K^-_0(x) R_{0N}(\tfrac{x}{z_N}) \cdots R_{01}(\tfrac{x}{z_1}); \]
again we have suppressed the labels $1,\ldots,N$ of the factors representing state space.
Next, according to \eqref{defn:transfermatrixbdy} we define the transfer matrix $\ca T \in \Mer(V^{\otimes N})$ of the inhomogeneous Heisenberg spin chain with reflecting boundary conditions as
\begin{equation}
 \ca T(x;\bm z) := \Tr_0 K'_0(x) R_{01}(\tfrac{1}{xz_1})^{-1} \cdots R_{0N}(\tfrac{1}{xz_N})^{-1}  K^-_0(x) R_{0N}(\tfrac{x}{z_N}) \cdots R_{01}(\tfrac{x}{z_1}).
 \label{eqn:inhbtransfermatrix}
\end{equation}
By virtue of Theorem \ref{thm:commutingtransfermatrices} and Prop.\ \ref{spinchain} we have
\begin{cor}
\label{cor:commutinginhbtransfermatrices}
Let $R \in \Mer(V^{\otimes 2})^\times$, $K',K^- \in \Mer(V)$ and $\bm z \in (\C^\times)^N$.
Suppose that $R^{t_1} \in \Mer(V^{\otimes 2})^\times$ and \eqref{YBE}, \eqref{RRERK} and \eqref{DRERK} are satisfied.
Let $\ca T$ be defined by \eqref{eqn:inhbtransfermatrix}.
Then for generic $x,y,\bm z$ we have
\[ [ \ca T(x;\bm z), \ca T(y;\bm z)] = 0.\]
\end{cor}

Here we provide the explicit formula for the bqKZ transport matrix corresponding to the type C case of Cherednik's works \cite{Ch1,Ch2} (or, equivalently, following \cite{JKKMW}).
The ingredients are $R \in \Mer(V^{\otimes 2})^\times$ and $K^\pm \in \Mer(V)$.
Then the \emph{bqKZ transport matrices} are the operators $\ca{A}_i(\bm z;p) \in \End(V^{\otimes N})$ are given by
\begin{equation} \label{defn:btransportmatrix}
\begin{aligned}
\ca{A}_i(\bm z;p)=& R_{i \, i-1}(\tfrac{p z_i}{z_{i-1}}) \cdots R_{i 1}(\tfrac{p z_i}{z_1}) K^+_i(p^{1/2} z_i) R_{1i}(z_1 z_i) \cdots R_{i-1 \, i}(z_{i-1} z_i) \cdot \\
&  \cdot  R_{i+1 \, i}(z_{i+1} z_i) \cdots R_{Ni}(z_N z_i) K_i^-(z_i) R_{Ni}(\tfrac{z_N}{z_i})^{-1} \cdots R_{i+1 \, i}(\tfrac{z_{i+1}}{z_i})^{-1}. \hspace{-5pt}
\end{aligned}
\end{equation}
The bqKZ equations are given by the system
\begin{equation} \label{eqn:bqKZ} \hspace{30mm} f(p^{\epsilon_i} \bm z) = \ca{A}_i(\bm z;p) f(\bm z) \qquad \n{for } i=1,\ldots,N, \end{equation}
for meromorphic functions $f: (\C^\times)^N \to V^{\otimes N}$.
We have the following statement
\begin{prop}[E.g.\ \cite{Ch1,Ch2}] \label{prop:bqKZconsistency}
Suppose the datum $(R,K^+,K^-)$ satisfies \eqrefs{YBE}{RRERK} and construct the $\ca{A}_i$ as per \eqref{defn:btransportmatrix}.
Then the system \eqref{eqn:bqKZ} is consistent, viz. the $\ca{A}_i$ satisfy the conditions
\begin{equation} \label{eqn:bqKZconsistency}
\hspace{4mm} \ca{A}_i(p^{\epsilon_j} \bm z;p) \ca{A}_j(\bm z;p) = \ca{A}_j(p^{\epsilon_i} \bm z;p) \ca{A}_i(\bm z;p) \qquad \n{for } i,j=1,\ldots,N.
\end{equation}
\end{prop}

\begin{remark}
Note that Cor.\ \ref{cor:commutinginhbtransfermatrices} and Prop.\ \ref{prop:bqKZconsistency} require the same type of conditions; both require the YBE \eqref{YBE}, the RRE \eqref{RRERK} and one additional reflection equation (the DRE \eqref{DRERK} and the LRE \eqref{LRERK}, respectively).
This is a consequence from the new result Thm.\ \ref{thm:commutingtransfermatrices} and ties these two notions of integrability more closely together.
In Lemma \ref{lem:refleqns} we will see that the DRE and the LRE are in fact equivalent (for essentially all R-matrices) and then complete the connection in Thm.\ \ref{thm:connection}.
All of these statements show that there are clear parallels between systems with reflecting boundary conditions and periodic systems.
\end{remark}

\subsubsection{Relations between solutions of reflection equations}
In order to connect the boundary transfer matrices $\ca T(x;\bm z)$ to the bqKZ transport matrices $\ca{A}_i(\bm z;p)$, we need to relate solutions of the DRE \eqref{DRERK}, which appear in $\ca T(x;\bm z)$, to solutions of the LRE, which are used in $\ca{A}_i(\bm z;p)$.
Fix $R \in \Mer(V^{\otimes 2})^\times$.
First we introduce subsets of $\Mer(V)$ defined by the various reflection equations:
\begin{align*}
\n{Refl}^+(R) &:= \left\{ K^+ \in \Mer(V) \mid \text{the LRE } \eqref{LRERK} \text{ holds} \right\}, \\
\n{Refl}^-(R) &:= \left\{ K^- \in \Mer(V) \mid \text{the RRE } \eqref{RRERK} \text{ holds} \right\}, \\
\n{Refl}'(R)& :=  \left\{ K' \in \Mer(V) \mid \text{the DRE } \eqref{DRERK} \text{ holds} \right\}.
\end{align*}
There are several noteworthy relations between these three sets.

Given $J \in \GL(V)$, define the following bijection $\chi_J: \Mer(V)  \to \Mer(V)$:
\[ \chi_J(Y)(x) = J^{-1} Y(x) J \]
for $Y \in \Mer(V)$ and generic $x \in \C$; evidently $\chi_J^{-1} = \chi_{J^{-1}}$.
If for generic $x$ we have
\begin{equation} \label{eqn:Rsigmasigma} \qquad  R_{12}(x) J_1 J_2 = J_1 J_2 R_{21}(x), \end{equation}
then $\chi_J$ restricts to a bijection: $\n{Refl}^+(R) \to \n{Refl}^-(R)$.
Obviously, if $R$ is P-symmetric ($R_{21}=R_{12}$) \eqref{LRERK} and \eqref{RRERK} coincide and then $\chi_J$ permutes the set $\n{Refl}^+(R) = \n{Refl}^-(R)$.
In fact, $J = \Id_V$ is a solution of \eqref{eqn:Rsigmasigma}.

Also, given $M \in \GL(V)$ and $r \in \C^\times$, define the following bijection $\psi_{M,r}: \Mer(V)^\times \to \Mer(V)^\times$:
\[ \psi_{M,r}(Y)(x) = Y(rx)^{-1}M, \qquad \n{with } \psi_{M,r}^{-1}(Y)(x) =  M Y(x/r)^{-1}.\]
Provided crossing symmetry \eqrefs{CS}{RMM} is satisfied, $\psi_{M,r}$ restricts to a bijection: $\n{Refl}^-(R) \cap \Mer(V)^\times \to \n{Refl}'(R) \cap \Mer(V)^\times $, as observed by \cite{Sk,MN,FSHY}.

There is also a (more elaborate) bijection\footnote{This bijection is related to Sklyanin's ``less obvious isomorphism'' \cite[Remark 2]{Sk} for P-symmetric R-matrices} from $\n{Refl}'(R)$ to $\n{Refl}^+(R)$.
More precisely, we will show that there exists an invertible $\C$-linear map $\phi_R: \Mer(V)  \to \Mer(V)$ which restricts to a bijection from $\n{Refl}'(R)$ to $\n{Refl}^+(R)$.
Namely, for $Y \in \Mer(V)$ define $\phi_R(Y) \in \Mer(V)$ by
\[
\phi_R(Y)_1(x) = \Tr_0 Y_0(x) P_{01} R_{01}(x^2)  \qquad \n{for generic } x. \]

Recall the notation $\tilde R$ defined in \eqref{eqn:Rtilde} relevant to the DRE \eqref{DRERK}.
\begin{lem} \label{lem:bijection}
Suppose we have $R \in \Mer(V^{\otimes 2})$ such that $R^{t_1} \in \Mer(V^{\otimes 2})^\times$.
Then $\phi_R$ is bijective; in fact $\phi_R^{-1} = \phi_{\tilde R}$.
\end{lem}

\begin{proof}
Let $Y \in \Mer(V)$ and $x \in \C$ generic.
Owing to \eqref{productoftraces} we have
\begin{align*}
 (\phi_R \circ \phi_{\tilde R} )(Y)_1(x) &= \Tr_0  \phi_{\tilde R}(Y)_0(x) P_{01} R_{01}(x^2)\\
&  = \Tr_{0,0'} Y_{0'}(x) P_{0'0} \tilde R_{0'0}(x^2) P_{01} R_{01}(x^2) \\
&  = \Tr_{0'} P_{0'1} Y_1(x) \Tr_{0} P_{00'} \tilde R_{0'1}(x^2)  R_{01}(x^2).
\end{align*}
Now note that by applying \eqref{Pintrace} we obtain
\[  (\phi_R \circ \phi_{\tilde R} )(Y)_1(x)  = \Tr_{0'} P_{0'1} Y_1(x) (\tilde R_{0'1}(x^2)^{t_{0'}}  R_{0'1}(x^2)^{t_{0'}})^{t_{0'}}  =  Y_1(x), \]
i.e.\ $\phi_R \circ \phi_{\tilde R} = \Id_V$; since $\tilde{\tilde R}=R$ and $R^{t_1} \in \Mer(V^{\otimes 2})^\times$ precisely if $\tilde R^{t_1} \in \Mer(V^{\otimes 2})^\times$, we immediately obtain $\phi_{\tilde R} \circ \phi_R = \Id_V$, which completes the proof.
\end{proof}

\begin{lem} \label{lem:refleqns}
Suppose we have $R \in \Mer(V^{\otimes 2})^\times$ such that $R^{t_1} \in \Mer(V^{\otimes 2})^\times$ and the YBE \eqref{YBE} is satisfied.
Then $\phi_R$ restricts to a bijection: $\n{Refl}'(R) \to \n{Refl}^+(R)$.
\end{lem}
The long and technical proof of this lemma is given in Appendix \ref{sec:proofs}.
Lemmas \ref{lem:bijection} and \ref{lem:refleqns} will play a key role in the main Theorem \ref{thm:connection} of this paper, linking interpolants of inhomogeneous boundary transfer matrices to bqKZ transport matrices.

\begin{figure}[h]
\caption{The bijections $\phi_R$, $\chi_J$ and $\psi_{M,r}$ relating the solutions of the three reflection equations with the necessary assumptions on $R$.
Strictly speaking, $\psi_{M,r}$ maps between subsets consisting of generically invertible solutions of appropriate reflection equations. }
\label{fig:diagram}
\includegraphics[width=170mm,clip=true,trim=50mm 220mm 0 35mm]{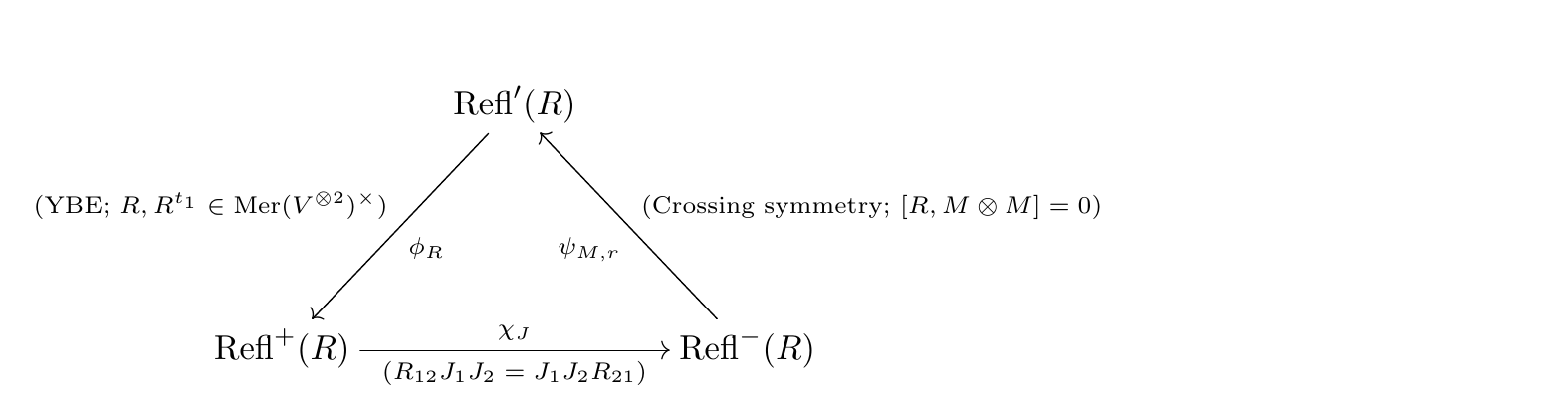}
\end{figure}
The bijections $\phi_R$, $\chi_J$ and $\psi_{M,r}$ are presented diagrammatically in Fig.\ \ref{fig:diagram}.
Evidently if $R$ satisfies all conditions \eqref{YBE}, \eqrefs{CS}{RMM}, \eqref{eqn:Rsigmasigma} and $R, R^{t_1} \in \Mer(V^{\otimes 2})^\times$ we can use this diagram to find a nontrivial bijection from, say, $\n{Refl}^+(R)$ to itself; i.e.\ if $K^{[0]}$ satisfies \eqref{LRERK} then so does $K^{[1]}:=(\phi_{R} \circ \psi_{M,r} \circ  \chi_J )(K^{[0]})$ given by
\[ K^{[1]}_1(x) = \Tr_0 J_0^{-1} K^{[0]}_0(rx)^{-1} J_0 M_0 P_{01} R_{01}(x^2). \]
Such a relation between solutions of a RE is known as a \emph{boundary crossing symmetry}.
E.g.\ for R- and K-matrices associated to the fundamental (vector) representation of $U_q(\hat{\f{sl}}_2)$ see \cite[Eqn. (3.25)]{GZ}, \cite[Eqn. (2.10)]{RSV} and \cite[Eqn. (4.10)]{SV}.

\subsubsection{The main theorem}

The following theorem is the analogon of Thm.\ \ref{thm:connectionper} in the case of reflecting quantum integrable systems.
It provides connections to the bqKZ transport matrices for the interpolants $\ca T(z_i;\bm z)$ and, contrary to the periodic case, also for the interpolants $\ca T(z_i^{-1};\bm z)$.

\emph{Boundary unitarity} for $K \in \Mer(V)$ is the condition that, for generic $x \in \C$,
\begin{equation} \label{Kunitary}
K(x) K(x^{-1}) \propto \Id_V.
\end{equation}
We have
\begin{thm} \label{thm:connection}
Let $R \in \Mer(V^{\otimes 2})^\times$, $K^+,K^- \in \Mer(V)$ and $\bm z \in (\C^\times)^N$.
Assume the YBE \eqref{YBE} and the regularity condition \eqref{Rregular}.
Let $\ca{A}_i$ be given by \eqref{defn:btransportmatrix} and $\ca T$ by \eqref{eqn:inhbtransfermatrix}, where we have written $K'= \phi_{\tilde R}(K^+)$.
Then, for generic values of $\bm z$,
\begin{equation} \label{eqn:transfermatrixqKZmatrix1} \ca T(z_i;\bm z) \propto \ca{A}_i(\bm z;1). \end{equation}
Furthermore, if $K^+,K^-$ satisfy \eqref{Kunitary} then, for generic values of $\bm z$,
\begin{equation} \label{eqn:transfermatrixqKZmatrix2} \ca T(z_i^{-1};\bm z) \propto \ca{A}_i(\bm z;1)^{-1}. \end{equation}
If in addition the REs \eqrefs{LRERK}{RRERK} are satisfied we have $[\ca T(z_i^{\pm 1} ;\bm z),\ca T(z_j^{\pm 1};\bm z)]=0$ for all $1 \leq i,j \leq N$ and all sign choices.
\end{thm}

\begin{proof}
We have, owing to \eqref{Rregular},
\begin{align*}
\ca T(z_i;\bm z) & \propto  \Tr_0 K'_0(z_i) R_{01}(\tfrac{1}{z_i z_1})^{-1} \cdots R_{0\, i\!-\!1}(\tfrac{1}{z_i z_{i\!-\!1}})^{-1} R_{0i}(\tfrac{1}{z_i^2})^{-1} \cdot \\
& \qquad\qquad \cdot  R_{0\, i\!+\!1}(\tfrac{1}{z_i z_{i\!+\!1}})^{-1} \cdots R_{0N}(\tfrac{1}{z_i z_N})^{-1} K^-_0(z_i) \cdot \\
& \qquad\qquad \cdot R_{0N}(\tfrac{z_i}{z_N}) \cdots R_{0 \, i\!+\!1}(\tfrac{z_i}{z_{i+1}}) P_{0i} R_{0 \, i\!-\!1}(\tfrac{z_i}{z_{i\!-\!1}}) \cdots R_{01}(\tfrac{z_i}{z_1}).
\end{align*}
Moving the factors $ P_{0i}$ and $R_{0 \, i\!-\!1}(\tfrac{z_i}{z_{i\!-\!1}}) \cdots R_{01}(\tfrac{z_i}{z_1})$ to the left and applying \eqref{Runitary} (a consequence of \eqref{YBE} and \eqref{Rregular}) we have
\begin{align*}
\ca T(z_i;\bm z) & \propto  \Tr_0 K'_0(z_i) P_{0i} R_{1i}(z_i z_1) \cdots R_{ i\!-\!1 \, i}(z_i  z_{i-1})  R_{0i}(z_i^2)  R_{0 \, i\!-\!1}(\tfrac{z_i}{z_{i\!-\!1}}) \cdots R_{01}(\tfrac{z_i}{z_1})  \cdot \\
& \qquad \cdot R_{ i\!+\!1 \, i}(z_i  z_{i+1}) \cdots R_{Ni}(z_i z_N) K^-_i(z_i) R_{Ni}(\tfrac{z_N}{z_i})^{-1} \cdots R_{ i\!+\!1 \, i}(\tfrac{z_{i+1}}{z_i})^{-1}.
\end{align*}
Now applying the YBE \eqref{YBE} repeatedly in the first line of this expression and moving various factors through $P$, we obtain
\begin{align*}
\ca T(z_i;\bm z) & \propto  R_{i \, i\!-\!1}(\tfrac{z_i}{z_{i\!-\!1}}) \cdots R_{i1}(\tfrac{z_i}{z_1})  \Bigl( \Tr_0  K'_0(z_i) P_{0i} R_{0i}(z_i^2) \Bigr) \cdot \\
& \qquad \cdot  R_{1i}(z_i z_1) \cdots R_{ i\!-\!1 \, i}(z_i  z_{i-1}) R_{ i\!+\!1 \, i}(z_i  z_{i+1}) \cdots R_{Ni}(z_i z_N) \cdot \\
& \qquad \cdot  K^-_i(z_i) R_{Ni}(\tfrac{z_N}{z_i})^{-1} \cdots R_{ i\!+\!1 \, i}(\tfrac{z_{i+1}}{z_i})^{-1}.
\end{align*}
Using Lemma \ref{lem:bijection} we recognize the partial trace as $\phi_R(K')_i(z_i) = K^+_i(z_i)$ and we obtain \eqref{eqn:transfermatrixqKZmatrix1}.
\eqref{eqn:transfermatrixqKZmatrix2} is obtained in a similar fashion, but in this case we also need boundary unitarity \eqref{Kunitary} for $K^\pm$ to match the inverted K-matrices in $\ca{A}_i(\bm z;1)^{-1}$ to the non-inverted ones in $\ca T(z_i^{-1};\bm z)$.
For the final commutativity statement we note that the conditions for Prop.\ \ref{prop:bqKZconsistency} and Lemma \ref{lem:refleqns} are satisfied.
\end{proof}

The most important results of this paper for spin chains are the two main quantum integrablity conditions (qKZ consistency conditions and transfer matrix commutativity) and the minimal conditions on the R-matrix datum required for these, as well asthe relation between the qKZ transport matrix and transfer matrix.
They are summarized for both periodic and reflecting systems in the following table. 

\begin{center}
\begin{tabular}{| m{16mm} | m{32mm} | m{32mm} | m{29mm} |}
\hline
Type of \newline boundary conditions & Commuting \newline transfer matrices & qKZ consistency \newline condition  &  Relation between \newline qKZ transport \newline matrix and \newline transfer matrix \\
\hline
\multirow{3}{*}{Periodic} & Cor.\ \ref{cor:commutinginhtransfermatricesper} & Prop.\ \ref{prop:qKZconsistencyper} & Thm.\ \ref{thm:connectionper} \\
& $[T(x;\bm z),T(y;\bm z)]=0$
& $A_i(p^{\epsilon_j}\bm z€;p) A_j(\bm z;p) = A_j(p^{\epsilon_i}\bm z;p) A_i(\bm z;p)$
& $T(z_i;\bm z) \propto A_i(\bm z;1)$ \\[2pt]
& {\small \begin{tabular}{ll} $\hspace{-2mm}  \bullet \hspace{-2mm} $ & YBE \\ $\hspace{-2mm}  \bullet \hspace{-2mm}$ & $[R(x),D \otimes D]=0$ \end{tabular} }
&  {\small \begin{tabular}{ll} $\hspace{-2mm}  \bullet \hspace{-2mm} $ & YBE \\ $\hspace{-2mm}  \bullet \hspace{-2mm}$ & $[R(x),D \otimes D]=0$ \end{tabular} }
& {\small \begin{tabular}{ll} $\hspace{-2mm}  \bullet \hspace{-2mm} $ & Regularity for $R$ \\ $\hspace{-2mm}  \bullet \hspace{-2mm}$ & Unitarity for $R$ \end{tabular} }   \\
\hline
\multirow{3}{*}{Reflecting} & Cor.\ \ref{cor:commutinginhbtransfermatrices} & Prop.\ \ref{prop:bqKZconsistency} & Thm.\ \ref{thm:connection} \\
& $\! [\ca T(x;\bm z),\ca T(y;\bm z)]=0\!$
& $\ca{A}_i(p^{\epsilon_j}\bm z;p) \ca{A}_j(\bm z;p) = \ca{A}_j(p^{\epsilon_i}\bm z;p) \ca{A}_i(\bm z;p)$
& $\ca T(z_i;\bm z) \propto \ca{A}_i(\bm z;1)$  \\
& {\small \begin{tabular}{ll} $\hspace{-2mm}  \bullet \hspace{-2mm} $ & YBE \\ $\hspace{-2mm}  \bullet \hspace{-2mm}$ & RRE, DRE \end{tabular} }
& {\small \begin{tabular}{ll} $\hspace{-2mm}  \bullet \hspace{-2mm} $ & YBE \\ $\hspace{-2mm}  \bullet \hspace{-2mm}$ & RRE, LRE \end{tabular} }
& {\small \begin{tabular}{ll} $\hspace{-2mm}  \bullet \hspace{-2mm}$ & Regularity for $R$ \\ $\hspace{-2mm}  \bullet \hspace{-2mm} $ & YBE \end{tabular} }  \\
\hline
\end{tabular}
\end{center}

\vspace{2mm}

Clearly there is now greater similarity in the required conditions on the R-matrix datum, both comparing between the type of boundary conditions (periodic vs. reflecting) and the type of integrability criterion (commuting transfer matrices vs. qZK consistency condition).

\begin{remark}
Note that, as in the periodic case, the condition \eqref{Runitary} is not necessary to derive the consistency condition $\ca{A}_i(p^{\epsilon_j} \bm z;p) \ca{A}_j(\bm z;p) = \ca{A}_j(p^{\epsilon_i} \bm z;p) \ca{A}_i(\bm z;p)$ or the commutativity $[\ca T(x;\bm z),\ca T(y;\bm z)]=0$; however it is required for the relation $\ca T(z_i;\bm z) \propto \ca{A}_i(\bm z;1)$, where it is a consequence of the explicitly assumed YBE \eqref{YBE} and the regularity condition \eqref{Rregular}.
\end{remark}

\section{The commutativity of transfer matrices revisited} \label{sec:tfermatsrevisited}
Thms. \ref{thm:connectionper} and \ref{thm:connection} can be wielded to recover the commutativity statements Thms. \ref{thm:commutingtransfermatricesper} and \ref{thm:commutingtransfermatrices} for inhomogeneous transfer matrices built up out of special classes of solutions of the local integrability conditions.
It is then possible (cf.\ \cite[Prop.\ 4]{RSZJ}) to deduce information about the ground state of certain quantum and statistical mechanics models if the parameter $p$ assumes special values.

The class of solutions we will have in mind is based on the image of the universal R-matrix of $U_q(\hat{\f{sl}}_n)$ in the tensor square of its fundamental (i.e.\ $n$-dimensional) representation.
In particular, let $V = \C^n$ for $n \in \Z_{>1}$ and consider the standard ordered orthonormal basis $(v_\alpha)_{\alpha=1}^n$ of $\C^N$, i.e.\ the $\alpha$-th entry of $v_\alpha$ is 1 and all other entries of $v_\alpha$ are 0.
Fix $q \in \C^\times$ and define $R \in \Mer(V^{\otimes 2})^\times$ by
\[ R(x) \cdot (v_\alpha \otimes v_\beta) = \begin{cases} q(1-x) v_\alpha \otimes v_\beta + (1-q^2)x v_\beta \otimes v_\alpha, & \alpha < \beta, \\ (1-q^2x) v_\alpha \otimes v_\alpha, & \alpha = \beta, \\  (1-q^2) v_\beta \otimes v_\alpha + q(1-x) v_\alpha \otimes v_\beta , & \alpha > \beta. \end{cases} \]
We immediately see that $R$ is a polynomial element of $\Mer(V^{\otimes 2})$ of degree 1 satisfying \eqref{Rregular}.
Using this we can directly check the YBE \eqref{YBE}; it is of course an existing result, cf.\ e.g.\ \cite{Bazh}.

\subsection{Periodic systems}

As an application of Thm.\ \ref{thm:connectionper} we will now present a novel proof of the commutativity of the inhomogeneous transfer matrices $t$ for special classes of R-matrices and associated solutions $D$ of \eqref{RDD}.

Let $R$ be as above and let $D \in \GL(V)$ be diagonal with respect to $(v_\alpha)_{\alpha=1}^n$; the compatibility condition \eqref{RDD} is satisfied.
The special case $n=2$ of this example corresponds to the datum $(R,D)$ for the periodic inhomogeneous Heisenberg XXZ spin-$\tfrac{1}{2}$ chain.

In addition to the conditions \eqref{YBE}, \eqref{RDD} and \eqref{Rregular} necessary for Prop.\ \ref{prop:qKZconsistencyper} and Thm.\ \ref{thm:connectionper}, we need further technical conditions to recover the commutativity of the transfer matrices.
Note that we have, for all $\alpha,\beta \in \{1,\ldots,n\}$,
\begin{align} \label{eqn:Rimage}
R(x)(v_\alpha \otimes v_\beta)  &\in \C v_\alpha \otimes v_\beta + \C v_\beta \otimes v_\alpha, \qquad \n{for all } x \in \dom(R), \\
\label{eqn:R0} R(0)(v_\alpha \otimes v_\beta) &\in \sum_{\gamma \leq \alpha, \; \delta \geq \beta} \C v_\gamma \otimes v_\delta, \\
\label{eqn:D} D v_\alpha &\in \C v_\alpha.
\end{align}

For this datum $(R,D)$ we can derive $[T(x;\bm z),T(y;\bm z)]=0$ from the consistency condition \eqref{eqn:qKZconsistency}.
\begin{thm} \label{thm:commutingtransfermatricesper2}
Let $R \in \Mer(V)$ and $D \in \End(V)$.
Assume that \eqref{YBE}, \eqref{RDD} and \eqref{Rregular} hold, that $R$ is a polynomial in $x$ of degree 1, and that, with respect to a certain ordered basis $(v_\alpha)_{\alpha=1}^n$ of $V$ we have \eqrefs{eqn:Rimage}{eqn:D}.
Then for all $x,y \in \C$ we have $[ T(x;\bm z), T(y;\bm z) ]=0$.
\end{thm}

\begin{remark}
Razumov, Stroganov and Zinn-Justin essentially established this statement for $n=2$ in \cite[Prop.\ 4]{RSZJ}.
There it is used to deduce information about the ground state of XXZ spin chains and Temperley-Lieb loop models in case the parameter $p$ assumes root-of-unity values.
\end{remark}

For the proof of Thm.\ \ref{thm:commutingtransfermatricesper2} it is helpful to write $v_{\bm \beta} = v_{\beta_1} \otimes \cdots \otimes v_{\beta_N}$ for an $N$-tuple $\bm \beta = (\beta_1,\ldots,\beta_N) \in \{1,\ldots,n\}^N$ and consider the decomposition
\[ V^{\otimes N} = \bigoplus_{\bm \alpha \atop \alpha_1 \leq \ldots \leq \alpha_N} W^S_{\bm \alpha}, \qquad \n{where} \qquad W^S_{\bm \alpha} := \bigoplus_{\bm \beta \in S_N(\bm \alpha)} v_{\bm \beta}. \]
Here, $S_N(\bm \alpha) = \{ (\alpha_{w 1},\ldots,\alpha_{w N}) \, \mid \, w \in S_N \}$ for $\bm \alpha = (\alpha_1,\ldots,\alpha_N)$ denotes the orbit of $\bm \alpha$ under the standard action of the symmetric group.

\begin{proof}[Proof of Thm.\ \ref{thm:commutingtransfermatricesper2}]
From \eqref{YBE}, \eqref{RDD} and \eqref{Rregular} we deduce that \eqref{eqn:qKZconsistency} holds true.
Also, Theorem \ref{thm:connection} applies; in particular we deduce that $[T(z_i;\bm z),T(z_j;\bm z)]=0$ for $1 \leq i,j \leq N$.

Because of the conditions on $R(x)$ and $D$, each $A_i(\bm z;p)$ preserves each subspace $W^S_{\bm \alpha}$.
On the other hand, Lemma \ref{lem:transfermatrixperatzero} applies, yielding that $T(0;\bm z)$ acts trivially on each $W^S_{\bm \alpha}$.
Hence, $[T(0;\bm z),A_i(\bm z;1)]=0$.

Combining these facts we see that $[T(x;\bm z),T(y;\bm z)]=0$ where $x$ and $y$ assume values in a collection of $N+1$ interpolation points
\[ Z := \{ z_1,\ldots,z_N,0 \}. \]
Since $T(x;\bm z)$ is a polynomial of degree $N$, we may conclude that, given $x \in Z$, the polynomials $ [T(x;\bm z),T(y;\bm z)]$, of degree $N$ in $y$, vanish for $N+1$ values of $y$; hence these polynomials are zero for all values of $y$, \emph{provided that $x \in Z$}.
Therefore for all $y \in \C$ we can draw the following conclusion: the polynomials $[T(x;\bm z),T(y;\bm z)]$, of degree $N$ in $x$, vanish for $N+1$ values of $x$, so that these polynomials are zero for all values of $x$.
The desired conclusion follows.
\end{proof}

\subsection{Reflecting systems}
The goal of this subsection is to give an alternative proof of the commutativity property $[\ca T(x;\bm z),\ca T(y,\bm z)]=0$ using the bqKZ consistency conditions for $p=1$, viz. $[\ca{A}_i(\bm z;1), \ca{A}_j(\bm z;1)]=0$, for R- and K-matrix datum satisfying further conditions, in analogy with the periodic case.
In particular, we will rely on the main Thm.\ \ref{thm:connection}.
The main example we will have in mind is given by the following datum $(R,K^+,K^-)$.
With $V = (v_\alpha)_{\alpha=1}^n$, let $R$ be as at the start of this section.
For $\alpha \in \{1,\ldots,n\}$, write $\bar \alpha := n+1-\alpha \in \{1,\ldots,n\}$.
Then $R$ satisfies \eqref{eqn:Rsigmasigma} with $J \in \GL(V)$ given by $J (v_\alpha) = v_{\bar \alpha}$.

Also, crossing symmetry \eqrefs{CS}{RMM} is satisfied, with $r = q^{h^\vee} = q^n$ and $M = \n{diag}(q^{2\rho}) = \n{diag}(q^{n-1},q^{n-3},\ldots,q^{-(n-1)})$, i.e.
\[ M( v_\alpha) =  q^{\bar \alpha - \alpha} v_\alpha, \qquad \alpha=1,\ldots,n. \]

Given $\theta, \kappa \in \C$ and $\bm \xi = (\xi_1,\ldots,\xi_{\lfloor n/2 \rfloor}) \in (\C^\times)^{\lfloor n/2 \rfloor}$ we define $K_{\theta,\kappa,\bm \xi} \in \Mer(V)^\times$ by
\[ K_{\theta,\kappa,\bm \xi}(x) \cdot v_\alpha =\theta x v_\alpha + \begin{cases}  (1-\kappa) v_\alpha + \tfrac{\kappa}{\xi_\alpha}(1-x^2) v_{\bar \alpha}, & \alpha < \tfrac{n+1}{2}, \\ (1-\kappa x^2) v_\alpha, & \alpha = \tfrac{n+1}{2}, \, n \n{ odd}, \\ \xi_{\bar \alpha}(1-x^2) v_{\bar \alpha} + (1-\kappa)x^2   v_\alpha, & \alpha > \tfrac{n+1}{2}.  \end{cases} \]
Evidently, $K_{\theta,\kappa,\bm \xi}$ is a polynomial element of $\Mer(V)$ of degree 2 satisfying \eqref{Kregular}.
Using these and properties satisfied by $R$, a straightforward argument shows that $K_{\theta,\kappa,\bm \xi}$ satisfies the LRE \eqref{LRERK} and boundary unitarity \eqref{Kunitary}.
Fix parameters $\theta^\pm, \kappa^\pm \in \C$ and $\bm \xi^\pm \in (\C^\times)^{\lfloor n/2 \rfloor}$ and now define $K^\pm \in \Mer(V)^\times$ by
\begin{align*}
K^+(x) &= (1-q^{2n} x^2) K_{\theta^+, \kappa^+,\bm \xi^+}(x) \\
K^- &= \chi_J(K_{\theta^-, \kappa^-,\bm \xi^-}).
\end{align*}
We immediately have that $K^+$ satisfies the LRE \eqref{LRERK} and, owing to the comment inmmediately following \eqref{eqn:Rsigmasigma}, that $K^-$ satisfies the RRE \eqref{RRERK}.
In addition, they both satisfy boundary regularity \eqref{Kregular} and boundary unitarity \eqref{Kunitary}.

Moreover, define $K' \in \Mer(V)^\times$ as in Thm.\ \ref{thm:connection} by $K'= \phi_{\tilde R}(K^+)$, so that $K'$ satisfies the DRE \eqref{DRERK}.
Explicitly it can be checked that, for generic $x$,
\[ K'(x)= \begin{cases} q^{-1} K_{\theta^+,\kappa^+,\tilde{\bm \xi}^+}(q^n x) M, & n \n{ even} \\ K_{q^{-1} \theta^+,q^{-2}\kappa^+,\Delta_q \bm \xi^+}(q^n x) M, & n \n{ odd}, \end{cases} \]
where
\[ \Delta_q = \n{diag}(q^{2 \lfloor n/2 \rfloor -1},q^{2 \lfloor n/2 \rfloor -3 },\ldots,q) \in \GL(\C^{\lfloor n/2 \rfloor}). \]

In addition to the identities just discussed, which are necessary for Lemma \ref{lem:transfermatrixatone}, Prop.\ \ref{prop:bqKZconsistency} and Thm.\ \ref{thm:connection}, we note that the datum $(R,K^+,K^-)$ satisfies further conditions.
We have
\begin{equation} \label{eqn:Kprimespecial} K'(\pm q^{-n}) \propto M. \end{equation}
Using the definitions of $K^-$ and $K'$ in terms of $K_{\theta,\kappa,\bm \xi}$ and the special form of $K_{\theta,\kappa,\bm \xi}(0)$, for all $\alpha \in \{1,\ldots,n\}$ we have
\begin{gather}
\label{eqn:Kimage} K^-(x)(v_\alpha) ,K'(x)(v_\alpha) \in \C v_\alpha + \C v_{\bar \alpha}, \qquad \n{for all } x \in \dom(K^-), \dom(K'),\\
\label{eqn:K0} K^-(0)(v_\alpha) \in \sum_{\gamma \leq \alpha} \C v_{\bar \gamma} \qquad \n{and} \qquad K'(0)(v_{\bar \alpha}) \in \sum_{\gamma \leq \alpha} \C v_{\gamma}.
\end{gather}
Finally, note that $K^-$ and $ K'$ are polynomial in $x$ of degree 2.

Before we state the Theorem, we need to address a subtlety which is absent from the periodic case.
In the periodic case we relied on the fact that the R-matrices were polynomial in the spectral parameter; however the boundary transfer matrix $\ca T(x;\bm z)$ contains inverses of R-matrices so that it cannot be polynomial and the analogon of the argument in the proof of Thm.\ \ref{thm:commutingtransfermatricesper2} will not apply.
Therefore we will work with the \emph{modified boundary transfer matrix} $\tilde{\ca T}$ defined by
\[ \tilde{\ca T}(x;\bm z) := \Tr_0 K'_0(x)  R_{10}(xz_1)\cdots R_{N0}(xz_N)  K^-_0(x) R_{0N}(\tfrac{x}{z_N}) \cdots R_{01}(\tfrac{x}{z_1}) \]
which is polynomial in $x$, provided the R- and K-matrices are.
We have $\tilde{\ca T}(x;\bm z) \propto \ca T(x;\bm z)$ for generic $\bm z$ provided unitarity \eqref{Runitary} is satisfied (the hidden factor in this proportionality relation only depends on the products $xz_i$).
In Thm.\ \ref{thm:connection} we have seen that, under suitable assumptions on the R- and K-matrix datum, the interpolants $\ca T(z_i^{\pm 1};\bm z)$, and hence also the modified interpolants $\tilde{\ca T}(z_i^{\pm 1};\bm z)$, are proportional to (inverted) bqKZ transport matrices for $p=1$:
\[\tilde{\ca T}(z_i;\bm z) \propto \ca{A}_i(\bm z;1), \qquad \tilde{\ca T}(z_i^{-1};\bm z) \propto \ca{A}_i(\bm z;1)^{-1} , \qquad i=1,\ldots,N. \]
Using these modified $\tilde{\ca T}$ we can derive $[\ca T(x;\bm z),\ca T(y;\bm z)]=0$ from the bqKZ consistency conditions \eqref{eqn:bqKZconsistency}, analogously to Thm.\ \ref{thm:commutingtransfermatricesper2}.

\begin{thm}\label{thm:commutingtransfermatrices2}
Let $R \in \Mer(V^{\otimes 2})^\times$ such that $R^{t_1} \in \Mer(V^{\otimes 2})^\times$, $K^+,K^- \in \Mer(V)^\times$, $M \in \GL(V)$, $r \in \C^\times$ and $\bm z \in (\C^\times)^N$.
Assume the conditions \eqrefs{YBE}{RRERK}, \eqrefs{CS}{RMM}, \eqref{Kregular}, \eqref{Rregular}, \eqref{eqn:Rsigmasigma}, \eqref{Kunitary} and $\pm r^{-1} \in \dom(K^-)$.
Write $K'= \phi_{\tilde R}(K^+)$ and assume that $\pm 1 \in \dom(K')$ and \eqref{eqn:Kprimespecial} are satisfied.
Assume that $R$, $K'$ and $K^-$ are polynomial in $x$ of degree 1, 2 and 2, respectively, and that, with respect to a certain ordered basis $(v_\alpha)_{\alpha=1}^n$ of $V$ we have \eqrefs{eqn:Rimage}{eqn:R0} and \eqrefs{eqn:Kimage}{eqn:K0}.
Then for all $x,y \in \dom(\ca T)$ we have $[ \ca T(x;\bm z), \ca T(y;\bm z) ]=0$.
\end{thm}

Before we give the proof, we consider another decomposition of $V^{\otimes N}$.
The hyperoctahedral group $\ca S_N$ acts on $\{1,\ldots,n\}^N$ by permutations and inversions of entries: when writing $\ca S_N = S_N \ltimes \langle e_1, \ldots,e_N \rangle$ we let elements from $S_N$ act by permutations and the $e_i$ by inversions:
\[ e_i(\alpha_1,\ldots,\alpha_N) = (\alpha_1,\ldots,\alpha_{i-1},\bar \alpha_i,\alpha_{i+1},\ldots,\alpha_N) \]
for $i=1,\ldots,N$ and $\alpha_1,\ldots,\alpha_N \in \{1,\ldots,n\}$.
Given an $N$-tuple $\bm \alpha = (\alpha_1,\ldots,\alpha_N) \in \{1,\ldots,n\}^N$, the orbit of $\bm \alpha$ under $\ca S_N$ is denoted
\[ \ca S_N(\bm \alpha) = \{ w(\bm \alpha) \, \mid \, w \in \ca S_N \} \subset \{1,\ldots, n\}^N. \]
We have the decomposition
\[ V^{\otimes N} = \bigoplus_{\bm \alpha \atop \alpha_1 \leq \ldots \leq \alpha_N \leq \tfrac{n+1}{2}}W^{\ca S}_{\bm \alpha}, \qquad \n{where} \qquad W^{\ca S}_{\bm \alpha} := \bigoplus_{\bm \beta \in \ca S_N(\bm \alpha)} v_{\bm \beta}. \]

\begin{proof}[Proof of Thm.\ \ref{thm:commutingtransfermatrices2}.]
Because \eqref{YBE} and \eqref{Rregular} are satisfied, we have \eqref{Runitary}.
Hence, $\tilde{\ca T}(x;\bm z) \propto \ca T(x;\bm z)$ for generic values of $\bm z$.
Since \eqrefs{YBE}{RRERK} hold true, the bqKZ consistency conditions \eqref{eqn:bqKZconsistency} are satisfied.
Hence,
\[ [\ca{A}_i(\bm z;1),\ca{A}_j(\bm z;1)]=[\ca{A}_i(\bm z;1),\ca{A}_j(\bm z;1)^{-1}]=[\ca{A}_i(\bm z;1)^{-1},\ca{A}_j(\bm z;1)^{-1}]=0 \]
for all $i,j \in \{1,\ldots,N\}$.
Consider the modified transfer matrix $\tilde{\ca T}(x;\bm z)$, which is a polynomial in $x$ of degree $2N+4$.
Theorem \ref{thm:connection}, combined with Lemmas \ref{lem:transfermatrixatone} and \ref{lem:transfermatrixatqinverse} yields $[\tilde{\ca T}(x;\bm z),\tilde{\ca T}(y;\bm z)]=0$ where $x$ and $y$ run through $2N+4$ interpolation points $z_1^{\pm 1},\ldots,z_N^{\pm 1},\pm 1,\pm r^{-1}$.
The final necessary interpolation point is again 0, for which we invoke Lemma \ref{lem:transfermatrixatzero} and note that the $\ca A_i(\bm z;1)^{\pm 1}$ preserve each sector $W^{\ca S}_{\bm \alpha}$, which owes to \eqref{eqn:Rimage} and \eqref{eqn:Kimage}.
This yields $[\tilde{\ca T}(x;\bm z),\tilde{\ca T}(y;\bm z)]=0$ for $x,y \in \ca{Z}$, where
\[ \ca{Z} := \{ z_1,\ldots,z_N,z_1^{-1},\ldots,z_N^{-1},1,-1,r^{-1},-r^{-1},0 \}. \]
From a similar argument as the one concluding the proof of Thm.\ \ref{thm:commutingtransfermatricesper2}, it follows that the polynomials $[\tilde{\ca T}(x;\bm z),\tilde{\ca T}(y;\bm z)]$ are zero for all values of $x$ and $y$.
Hence also the original boundary transfer matrices commute: $[\ca T(x;\bm z), \ca T(y;\bm z) ]=0$ for generic values of $x$ and $y$.
Because $\ca T(x;\bm z)$ depends meromorphically on $x$, the desired statement follows.
\end{proof}

\section{Outlook} \label{sec:outlook}

The reflection equations and boundary transfer matrix as written down by Sklyanin \cite{Sk} are associated to reflecting integrable systems where the particle-boundary interaction does not change the nature of the particle: the states of the incoming and outgoing particle inhabit the same vector space.
It is equally possible (see \cite{ACDFR1,ACDFR2,Doikou1,Doikou2}) to consider the case where after reflection with the boundary the vector space related to the outgoing particle is dual to the one of the incoming particle, characterizing so-called \emph{soliton non-preserving} boundary conditions and subject to the so-called \emph{twisted reflection equation}.
Replacing \emph{both} K-matrices by solutions of suitable twisted reflection equations, we obtain transfer matrices and qKZ transport matrices that are still well-defined operators acting on a global state space.
It would be of interest to modify the theory in the present paper to cover this case as well.\\

With respect to the results of Section \ref{sec:inhtfermatsandscatmats}, one may consider the generalization of Thms. \ref{thm:connectionper} and \ref{thm:connection} to the case where in the state space $V_1 \otimes \cdots \otimes V_N$ not all $V_i$ are isomorphic.
Then some R-matrices making up the (b)qKZ transport matrix act in a tensor product of different spaces, and for them no direct analogon of regularity \eqref{Rregular} exists.
It would be interesting to see in how far the argument can be salvaged.
Also cf. point (2) below.\\

It should be possible to generalize the analysis in Section \ref{sec:tfermatsrevisited} in the following ways.
\begin{enumerate}
\item There are also polynomial solutions $K$ of the DRE or RRE of degree 1 (for $R$ of $U_q(\hat{\f{sl}}_n)$-type).
In this case only one of $K(1), K(-1)$ is a multiple of the identity.
However, since the degree of $\ca T$ is reduced by one, we need one fewer interpolation point.
Moreover, there are also constant (off-diagonal) solutions $K$ to the REs, which do not satisfy boundary regularity; however, in this case the degree of $\ca T$ is reduced by two, and we do not need these interpolation points at all.
Hence, conjecturally, the argument of Section \ref{sec:tfermatsrevisited} can be modified to deal with these cases, as well.
A special case of this would be the qKZ equations associated to affine root systems of types B and D (in Cherednik's framework), which correspond to special choices of constant K-matrices (for $K^+$ in type B and for both $K^\pm$ in type D). 
\item We can also look at non-fundamental finite-dimensional representations of $U_q(\hat{\f{sl}}_n)$, in which case R- and K-matrices can be obtained from the ones discussed here through \emph{fusion} (see, e.g., \cite{DoikouFusion,MN2, RSV2}).
Then the degree of $R$ will be higher, so more interpolation points are needed.
For the periodic case, this problem was considered in \cite{FZJ}, where additional interpolation points of the transfer matrix were found, related to the $q$-dependent shift in the spectral parameter associated with fusion.
This could be extended to the reflecting case.
\item The R- and K-matrices from Section \ref{sec:tfermatsrevisited} are gauge-equivalent to the trigonometric solutions of the additive YBE and REs.
It is equally natural to do a similar analysis with rational or elliptic R- and K-matrices.
\item It would be interesting to consider R- and K-matrices associated to $U_q(\hat{\f g})$ for the orthogonal and symplectic Lie-algebras $\f g$.
\end{enumerate}

\appendix

\section{Linear algebra in tensor products} \label{sec:linearalgebra}

Consider a tensor product $\oplus_{i \in I} V_i =: V_I$ of complex vector spaces $V_i$, where $I$ is a finite ordered set (typically, $\{1,\ldots,N\}$ or $\{0,1,\ldots,N\}$).
We use standard subscript ``tensor leg notation'' to turn a local operator (i.e.\ an operator acting on one or two of the $V_i$) into a global operator (one acting on $V_I$) by stipulating that it acts nontrivially only in those $V_i$ specified by the subscript.
Consider the canonical embeddings $\lambda_i: \End(V_i) \to \End(V_I)$ ($i \in I$) and $\mu_{ij}: \End(V_i \otimes V_j) \to \End(V_I)$ ($i,j \in I, i \ne j$). 
Then for $i \in I$ and $Y \in \End(V_i)$ consider $Y_i := \lambda_i(Y) = \Id_{V_{<i}} \otimes Y \otimes \Id_{V_{>i}} \in \End(V_I)$; in other words it is the linear operator on $V_I$ that acts trivially in all $V_j$ where $j \ne i$ and as $Y$ in $V_i$.
Similarly, for $i,j \in I$, $i \ne j$ and $X \in \End(V_i \otimes V_j)$, we define $X_{ij} = \mu_{ij}(X) \in \End(V_I)$; it is the linear operator on $V_I$ which acts trivially in all $V_k$ where $k \ne i,j$ and as $X$ in $V_i \otimes V_j$.

Furthermore, given parameter-dependent local operators $X \in \Mer(V_{ij}), Y \in \Mer(V_i)$ we define $Y_i, X_{ij} \in \Mer(V_I)$ by $X_{ij}(x) = (X(x))_{ij}$ for $x \in \dom(X)$ and $Y_i(x) = (Y(x))_i$ for $x \in \dom(Y)$.\\

Partial transposition, i.e.\ transposing $X \in \End(V_I)$ with respect to a finite-dimensional $V_i$ ($i \in I$) is denoted by $X^{t_i} \in \End(V_I)$.
More precisely, the partial transpose in $V_I$ with respect to $V_i$ is the unique linear operator $t_i: \End(V_I) \to \End(V_I): Z \mapsto Z^{t_i}$ such that
\[ (X \otimes Y \otimes \tilde X)^{t_i} = X \otimes Y^t \otimes \tilde X, \quad \n{for all } X \in \End(V_{<i}), \, Y \in V_i, \, \tilde X \in \End(V_{>i}). \]
Whenever we partially transpose a linear operator, the pertinent vector space is saliently assumed to be finite-dimensional.

Let $X \in \End(V_i \otimes V_j)$ and $\tilde X \in \End(V_i \otimes V_k)$.
Then in $\End(V_I)$ we have
\begin{equation}\label{transposeofproduct}
(X_{ij} \tilde X_{ik})^{t_i} = \tilde X_{ik}^{t_i} X_{ij}^{t_i}, \qquad (X_{ij} \tilde X_{ik})^{t_j} = X_{ij}^{t_j}\tilde X_{ik}, \qquad  (X_{ij} \tilde X_{ik})^{t_k} = X_{ij}\tilde X_{ik}^{t_k} \hspace{-4mm}.
\end{equation}

Furthermore, the notion of taking the trace of $X \in \End(V_I)$ with respect to $V_i$ (``partial trace'') is denoted $\Tr_i(X)$ or $\Tr_{V_i}(X)$.
More precisely, if $V_i$ for some $i \in I$ is finite-dimensional, then $\Tr_i$ is the unique linear operator: $\End(V_I) \to \End(V_{I \setminus \{i\}})$ such that
\[ \Tr_i (X \otimes Y \otimes \tilde X) = \Tr(Y) X \otimes \tilde X, \quad \n{for all } X \in \End(V_{<i}), \, Y \in V_i, \, \tilde X \in \End(V_{>i}). \]
Whenever we take the partial trace of a linear operator, the pertinent vector space is saliently assumed to be finite-dimensional.
We have $\Tr_j P_{ij} = \Id_{V_i}$ if $V_i = V_j$ and
\begin{equation} \Tr_i Z Y_i= \Tr_i Y_i Z \in \End(V_{I \setminus \{i\}}), \qquad \n{for all } Y \in \End(V_i), \, Z \in \End(V_I). \label{commuteinsidetrace} \end{equation}
For $i, j \in I$ with $i \ne j$ we can consecutively take partial traces with respect to $V_i$ and $V_j$; the order of this does not matter and we employ the notation
\[ \Tr_{i,j} := \Tr_i \Tr_j = \Tr_j \Tr_i : \End(V_I) \to \End(V_{I \setminus\{ i,j\}}). \]
We have the identity
\begin{equation} \Tr_{i,j} X_{ik} \tilde X_{jk}  = (\Tr_i X_{ik}) (\Tr_j \tilde X_{jk}) \; \in \End(V_{I \setminus \{i,j\}}), \label{productoftraces} \end{equation}
for all $X \in \End(V_i \otimes V_k)$, $\tilde X \in \End(V_j \otimes V_k)$; note that $\Tr_i X_{ik}$ acts trivially in $V_j$ so that we may view it as an element of $\End(V_{I \setminus \{j,k\}})$.

The interplay between partial traces and partial transposes is captured by the following identities in $\End(V_{I \setminus \{i\}})$ :
\begin{align}
\Tr_i Z^{t_i} \tilde Z^{t_i} &= \Tr_i Z \tilde Z, && \n{for all } Z,\tilde Z \in \End(V_I),  \label{traceoftransposes} \\
\Tr_i \left(Z^{t_j}\right) &= (\Tr_i Z)^{t_j}, && \n{for all } Z \in \End(V_I), \quad j \in I\setminus\{i\}. \label{traceofothertranspose}
\end{align}
Combining \eqref{transposeofproduct} and \eqref{traceofothertranspose}, for $V_i \cong V_j$ and $X, \tilde X \in \End(V_j \otimes V_k)$, we obtain the useful identity
\begin{equation}
\Tr_i P_{ij} X_{jk} \tilde X_{ik} = (X_{jk}^{t_j} \tilde X_{jk}^{t_j})^{t_j} \quad \in \End(V_{I \setminus \{i\}}). \label{Pintrace}
\end{equation}
All identities in this appendix naturally transform to identities for objects in $\Mer(V_I)$ and are used as such in the main text. 

\section{Proof of Lemma \ref{lem:refleqns}} \label{sec:proofs}

\begin{proof}
First we will show that $\phi_R(\n{Refl}'(R)) \subset \n{Refl}^+(R)$, in other words derive
\[  R_{12}(x/y) \phi_R(K')_1(x) R_{21}(xy) \phi_R(K')_2(y)  = \phi_R(K')_2(y) R_{12}(xy) \phi_R(K')_1(x) R_{21}(x/y) \]
from the DRE \eqref{DRERK} for generic values of $x,y$.
Owing to \eqref{traceoftransposes} we have
\[  \phi_R(K')_1(x) =\Tr_0 \bigl(K'_0(x) P_{01}\bigr)^{t_0} R_{01}(x^2)^{t_0}. \]
Owing to \eqref{productoftraces}, the identity $\tilde R_{00'}(xy)^{t_0} R_{00'}(xy)^{t_0} = \Id_{V^{\otimes 2}}$ and \eqref{transposeofproduct} we have
\begin{align*}
&  \phi_R(K')_1(x) R_{21}(xy) \phi_R(K')_2(y) = \\
& = \Tr_{0} \bigl(K'_0(x) P_{01}\bigr)^{t_0} R_{01}(x^2)^{t_0} R_{21}(xy)  \Tr_{0'} K'_{0'}(y) P_{0'2} R_{0'2}(y^2) \displaybreak[2] \\
& = \Tr_{0,0'}  K'_{0'}(y) \bigl(K'_0(x) P_{01}\bigr)^{t_0} P_{0'2}  R_{01}(x^2)^{t_0} R_{0'1}(xy)   R_{0'2}(y^2) \displaybreak[2] \\
&= \Tr_{0,0'}  K'_{0'}(y)  \bigl(K'_0(x) P_{01}\bigr)^{t_0} \tilde R_{00'}(xy)^{t_0} P_{0'2} R_{02}(x y)^{t_0} R_{01}(x^2)^{t_0} R_{0'1}(xy)   R_{0'2}(y^2) \displaybreak[2] \\
&= \Tr_{0,0'} \bigl(  K'_{0'}(y) \tilde R_{00'}(xy) K'_0(x) P_{01} P_{0'2}\bigr)^{t_0}  \bigl(R_{01}(x^2)  R_{02}(x y) R_{0'1}(xy)   R_{0'2}(y^2) \bigr)^{t_0}  \displaybreak[2]
\end{align*}
so that \eqref{traceoftransposes} leads to
\begin{equation} \label{eqn:expression}
\begin{aligned}
&  \phi_R(K')_1(x) R_{21}(xy) \phi_R(K')_2(y)  = \\
& \qquad =\Tr_{0,0'} K'_{0'}(y) \tilde R_{00'}(xy) K'_0(x) P_{01} P_{0'2}R_{01}(x^2)  R_{02}(x y) R_{0'1}(xy)   R_{0'2}(y^2)  \end{aligned}
\end{equation}
It follows that
\begin{align*}
& R_{12}(\tfrac{x}{y})  \phi_R(K')_1(x) R_{21}(xy) \phi_R(K')_2(y)  =  \\
& \qquad = \Tr_{0,0'} K'_{0'}(y) \tilde R_{00'}(xy) K'_0(x) P_{01} P_{0'2} R_{00'}(\tfrac{x}{y}) R_{01}(x^2)  R_{02}(x y) R_{0'1}(xy)   R_{0'2}(y^2)  \hspace{-3mm} \displaybreak[2] \\
& \qquad = \Tr_{0,0'} K'_{0'}(y) \tilde R_{00'}(xy) K'_0(x) P_{01} P_{0'2} R_{02}(x y) R_{01}(x^2)   R_{0'2}(y^2)  R_{0'1}(xy)  R_{00'}(\tfrac{x}{y}) \hspace{-3mm} \displaybreak[2] \\
& \qquad = \Tr_{0,0'} R_{00'}(\tfrac{x}{y}) K'_{0'}(y) \tilde R_{00'}(xy) K'_0(x) P_{01} P_{0'2} R_{02}(x y) R_{01}(x^2)   R_{0'2}(y^2)  R_{0'1}(xy) \hspace{-3mm} \displaybreak[2] .
\end{align*}
where we have applied \eqref{YBE} twice and \eqref{commuteinsidetrace}.
Now applying the DRE \eqref{DRERK} followed by \eqref{YBE} (twice) we obtain
\begin{align}
& R_{12}(\tfrac{x}{y})  \phi_R(K')_1(x) R_{21}(xy) \phi_R(K')_2(y)  = \nonumber \\
&= \Tr_{0,0'}  K'_0(x)\tilde R_{0'0}(xy) K'_{0'}(y)  R_{0'0}(\tfrac{x}{y})  P_{01} P_{0'2} R_{02}(x y) R_{01}(x^2)   R_{0'2}(y^2)  R_{0'1}(xy) \hspace{-3mm} \displaybreak[2] \nonumber \\
&= \Tr_{0,0'}  K'_0(x)\tilde R_{0'0}(xy) K'_{0'}(y) P_{01} P_{0'2}   R_{01}(x^2)   R_{02}(x y)  R_{0'1}(xy) R_{0'2}(y^2)   R_{21}(\tfrac{x}{y}) \hspace{-3mm} \displaybreak[2],
\end{align}
which equals $\phi_R(K')_2(y) R_{12}(xy) \phi_R(K')_1(x) R_{21}(x/y)$ as desired by virtue of \eqref{eqn:expression}.

It remains to show that $\phi_{\tilde R}(\n{Refl}^+(R)) \subset \n{Refl}'(R)$ which can be done in an analogous way as before, with the following modifications.
Instead of inserting $\tilde R_{00'}(xy)^{t_0} R_{00'}(xy)^{t_0} = \Id_{V^{\otimes 2}}$ we insert $R_{00'}(xy)^{t_0} \tilde R_{00'}(xy)^{t_0} = \Id_{V^{\otimes 2}}$.
Initially it leads to
\begin{equation} \label{eqn:expression2}
\begin{aligned}
&\phi_{\tilde R}(K^+)_1(x) \tilde R_{21}(xy) \phi_{\tilde R}(K^+)_2(y) =\\
& \qquad = \Tr_{0,0'} K^+_{0'}(y) R_{00'}(xy) K^+_0(x) P_{01} P_{0'2} \tilde R_{01}(x^2) \tilde R_{02}(xy) \tilde R_{0'1}(xy) \tilde R_{0'2}(y^2).
\end{aligned}
\end{equation}
Now we claim that, since $R$ satisfies the YBE \eqref{YBE}, for generic values $x,y$ we have
\begin{equation} \label{YBEtwisted}
R_{12}(\tfrac{x}{y})^{-1} \tilde R_{13}(x) \tilde R_{23}(y) = \tilde R_{23}(y) \tilde R_{13}(x)  R_{12}(\tfrac{x}{y})^{-1}, \end{equation}
which can be straightforwardly checked using \eqref{transposeofproduct}.
Repeated use of \eqref{YBEtwisted} instead of \eqref{YBE}, as well as applying \eqref{commuteinsidetrace} and \eqref{LRERK} now allows us to continue along the same lines as before:
\begin{align*}
&R_{12}(\tfrac{x}{y})^{-1} \phi_{\tilde R}(K^+)_1(x) \tilde R_{21}(xy) \phi_{\tilde R}(K^+)_2(y) =\\
& \quad = \Tr_{0,0'} K^+_{0'}(y) R_{00'}(xy) K^+_0(x) P_{01} P_{0'2} R_{00'}(\tfrac{x}{y})^{-1}  \tilde R_{01}(x^2) \tilde R_{02}(xy) \tilde R_{0'1}(xy) \tilde R_{0'2}(y^2) \displaybreak[2] \\
& \quad = \Tr_{0,0'} K^+_{0'}(y) R_{00'}(xy) K^+_0(x) P_{01} P_{0'2} \tilde R_{0'1}(xy) \tilde R_{01}(x^2)  \tilde R_{0'2}(y^2)\tilde R_{02}(xy)  R_{00'}(\tfrac{x}{y})^{-1} \displaybreak[2] \\
& \quad = \Tr_{0,0'}  R_{00'}(\tfrac{x}{y})^{-1} K^+_{0'}(y) R_{00'}(xy) K^+_0(x) P_{01} P_{0'2} \tilde R_{0'1}(xy) \tilde R_{01}(x^2)  \tilde R_{0'2}(y^2)\tilde R_{02}(xy) \displaybreak[2]  \\
& \quad = \Tr_{0,0'} K^+_0(x)  R_{0'0}(xy)   K^+_{0'}(y) P_{01} P_{0'2} R_{21}(\tfrac{x}{y})^{-1}  \tilde R_{0'1}(xy) \tilde R_{01}(x^2)  \tilde R_{0'2}(y^2)\tilde R_{02}(xy) \displaybreak[2]  \\
& \quad = \Tr_{0,0'} K^+_0(x)  R_{0'0}(xy)   K^+_{0'}(y) P_{01} P_{0'2} \tilde R_{0'2}(y^2) \tilde R_{0'1}(xy)   \tilde R_{02}(xy) \tilde R_{01}(x^2) R_{21}(\tfrac{x}{y})^{-1}.
\end{align*}
By virtue of \eqref{eqn:expression2} the right-hand side of the last expression can be seen to equal $\phi_{\tilde R}(K^+)_2(y) \tilde R_{12}(xy) \phi_{\tilde R}(K^+)_1(x) R_{21}(\tfrac{x}{y})^{-1}$, yielding the desired conclusion.
\end{proof}

\section{Properties of R- and K-matrices associated to the fundamental representation of $U_q(\hat{\f{sl}}_n)$}

For the next two lemmas, denote the matrix entries of $\bar R := R(0)$, $D$, $\bar K := K^-(0)$ and $\bar L := K'(0)$ by $r_{\alpha \beta}^{\gamma \delta}, d_\alpha^\gamma, k_\alpha^\gamma, l_\alpha^\gamma \in \C$, respectively, for $\alpha,\beta,\gamma,\delta \in \{1,\ldots,n\}$, viz.
\begin{gather*}
\bar R(v_\alpha \otimes v_\beta) = \sum_{\gamma,\delta} r_{\alpha \beta}^{\gamma \delta} v_\gamma \otimes v_\delta, \qquad D(v_\alpha) = \sum_\gamma d_\alpha^\gamma v_\gamma, \\
\bar K(v_\alpha) = \sum_\gamma k_\alpha^\gamma v_\gamma, \qquad \bar L(v_\alpha) = \sum_\gamma l_\alpha^\gamma v_\gamma.
\end{gather*}

\begin{lem} \label{lem:transfermatrixperatzero}
Let $R \in \Mer(V)$ and $D \in \End(V)$.
Assume that $0 \in \dom(R)$ and that, with respect to a certain ordered basis $(v_\alpha)_{\alpha=1}^n$ of $V$, we have \eqref{eqn:R0} and \eqref{eqn:D}.
Then for all $N$-tuples $\bm \alpha$ satisfying $1 \leq \alpha_1 \leq \ldots \leq \alpha_N \leq n$, $T(0;\bm z)$ acts trivially on each subspace $W^S_{\bm \alpha}$, i.e.\ there exists $C_{\bm \alpha } \in \C$ such that
\[ \hspace{20mm} T(0;\bm z) (v_{\bm \beta}) = C_{\bm \alpha} v_{\bm \beta}, \qquad \n{for } v_{\bm \beta} \in W^S_{\bm \alpha}. \]
\end{lem}

\begin{proof}
The condition \eqref{eqn:R0} can be generalized by induction with respect to $N$ to
\begin{equation} \label{eqn:U0} \bar R_{0N} \cdots \bar R_{01}(v_\alpha \otimes v_{\bm \beta}) \in \Bigl( \prod_{i=1}^N r_{\alpha \beta_i}^{\alpha \beta_i} \Bigr) v_\alpha \otimes v_{\bm \beta} +  \sum_{\gamma < \alpha} v_\gamma \otimes V^{\otimes N} \end{equation}
for $\bm \beta \in \{1,\ldots,n\}^N$.
Hence
\[ D_0 \bar R_{0N} \cdots \bar R_{01} (v_\alpha \otimes v_{\bm \beta} ) \in d_\alpha^\alpha \Bigl( \prod_{i=1}^N r_{\alpha \beta_i}^{\alpha \beta_i} \Bigr) v_\alpha \otimes v_{\bm \beta} +  \sum_{\gamma < \alpha} v_\gamma \otimes V^{\otimes N} \]
and
\[ T(0;\bm z)(v_{\bm \beta}) = \biggl( \sum_{\alpha=1}^n d_\alpha^\alpha  \prod_{i=1}^N r_{\alpha \beta_i}^{\alpha \beta_i}  \biggr) v_{\bm \beta}. \]
Evidently the coefficient in front of $v_{\bm \beta}$ is unchanged if we permute the $\beta_i$.
\end{proof}

\begin{lem} \label{lem:transfermatrixatzero}
Let $R \in \Mer(V^{\otimes 2})^\times$, $K',K^- \in \Mer(V)^\times$ and $\bm z \in (\C^\times)^N$.
Suppose that there exists $J \in \GL(V)$ such that \eqref{eqn:Rsigmasigma} holds.
Assume that 0 is in the domains of $R$, $K'$ and $K^-$ and that, with respect to a certain ordered basis $(v_\alpha)_{\alpha=1}^n$ of $V$, we have \eqref{eqn:R0} and \eqref{eqn:K0}.
Then for all $N$-tuples $\bm \alpha$ satisfying $1 \leq \alpha_1 \leq \ldots \leq \alpha_N \leq \tfrac{n+1}{2}$, $\tilde{\ca T}(0;\bm z)$ acts trivially on each subspace $W^{\ca S}_{\bm \alpha}$, i.e.\ there exists $\ca C_{\bm \alpha } \in \C$ such that
\[ \hspace{20mm} \tilde{\ca T}(0;\bm z) (v_{\bm \beta}) = \ca{C}_{\bm \alpha} v_{\bm \beta}, \qquad \n{if } v_{\bm \beta} \in W^{\ca S}_{\bm \alpha}. \]
\end{lem}

\begin{proof}
Recall from the proof of Lemma \ref{lem:transfermatrixperatzero} that \eqref{eqn:R0} implies \eqref{eqn:U0}.
In the same way we can derive
\[ \bar R_{10} \cdots \bar R_{N0}(v_{\bar \alpha} \otimes v_{\bm \beta}) \in  \Bigl( \prod_{i=1}^N r_{\beta_i \bar \alpha }^{\beta_i \bar \alpha } \Bigr)  v_{\bar \alpha} \otimes v_{\bm \beta} + \sum_{\gamma < \alpha} v_{\bar \gamma} \otimes V^{\otimes N}, \]
where owing to \eqref{eqn:Rsigmasigma} we may re-write each $r_{\beta_i \bar \alpha }^{\beta_i \bar \alpha }$ as $r_{\alpha \bar \beta_i }^{\alpha \bar \beta_i }$.
Combining the above remarks with \eqref{eqn:K0} we obtain
\begin{align*}
\bar K_0 \bar R_{0N} \cdots \bar R_{01}(v_\alpha \otimes v_{\bm \beta}) &\in k_\alpha^{\bar \alpha}  \Bigl( \prod_{i=1}^N r_{\alpha \beta_i}^{\alpha \beta_i} \Bigr) v_{\bar \alpha} \otimes v_{\bm \beta} +  \sum_{\gamma < \alpha} v_{\bar \gamma} \otimes V^{\otimes N}, \displaybreak[2] \\
\bar L_0 \bar R_{10} \cdots \bar R_{N0}(v_{\bar \alpha} \otimes v_{\bm \beta} ) & \in  l_{\bar \alpha}^\alpha \Bigl( \prod_{i=1}^N r_{\alpha \bar \beta_i }^{\alpha \bar \beta_i } \Bigr)  v_{ \alpha} \otimes v_{\bm \beta} + \sum_{\gamma < \alpha} v_{ \gamma} \otimes V^{\otimes N},
\end{align*}
so that
\begin{align*}
& \bar L_0 \bar R_{10} \cdots \bar R_{N0} \bar K_0 \bar R_{0N} \cdots \bar R_{01}(v_\alpha \otimes v_{\bm \beta}) \in \\
&  \qquad  k_\alpha^{\bar \alpha} l_{\bar \alpha}^\alpha \Bigl( \prod_{i=1}^N r_{\alpha \beta_i}^{\alpha \beta_i}  r_{\alpha \bar \beta_i }^{\alpha \bar \beta_i } \Bigr) v_{ \alpha} \otimes v_{\bm \beta} +  \sum_{\gamma < \alpha} v_{ \gamma} \otimes V^{\otimes N}.
\end{align*}
Hence any $v_{\bm \beta}$ is an eigenfunction of $\tilde{\ca T}(0;\bm z)$ with the eigenvalue invariant under the action of $\ca S_N$:
\[ \tilde{\ca T}(0;\bm z) (v_{\bm \beta}) = \biggl( \sum_{\alpha=1}^n  k_\alpha^{\bar \alpha} l_{\bar \alpha}^\alpha \Bigl( \prod_{i=1}^N r_{\alpha \beta_i}^{\alpha \beta_i} r_{\alpha \bar \beta_i }^{\alpha \bar \beta_i } \Bigr) \biggr) v_{\bm \beta}. \qedhere \]
\end{proof}

\begin{lem} \label{lem:transfermatrixatqinverse}
Let $R \in \Mer(V^{\otimes 2})^\times$, $K^-,K' \in \Mer(V)^\times$, $M \in \GL(V)$, $r \in \C^\times$ and $\bm z \in (\C^\times)^N$.
Let $\ca T$ be given by \eqref{eqn:inhbtransfermatrix}.
If \eqrefs{CS}{RMM} and \eqref{eqn:Kprimespecial} are satisfied and $\pm r^{-1} \in \dom(K^-)$ then
\[ \ca T(\pm r^{-1};\bm z) \propto  \bigl(\Tr K^-(\pm r^{-1})M\bigr) \Id_{V^{\otimes N}}. \]
\end{lem}

\begin{proof}
Using \eqref{eqn:Kprimespecial} we have
\begin{align*}
\ca T(\pm r^{-1};\bm z) &\propto \Tr_0 M_0 R_{01}(\pm r z_1^{-1})^{-1} \cdots R_{0N}(\pm r z_N^{-1})^{-1} \cdot \\
& \hspace{20mm} \cdot K^-_0(\pm r^{-1}) R_{0N}(\pm r^{-1} z_N^{-1}) \cdots R_{01}(\pm r^{-1} z_1^{-1}).
\end{align*}
Hence, applying \eqref{traceoftransposes} and \eqref{transposeofproduct} we have
\begin{align*}
\ca T(\pm r^{-1};\bm z)&= \Tr_0 \bigl(M_0 R_{01}(\pm r z_1^{-1})^{-1} \cdots R_{0N}(\pm r z_N^{-1})^{-1} K^-_0(\pm r^{-1})\bigr)^{t_0} \cdot \\
& \hspace{20mm} \cdot \bigl( R_{0N}(\pm r^{-1} z_N^{-1}) \cdots R_{01}(\pm r^{-1} z_1^{-1})\bigr)^{t_0}  \displaybreak[2] \\
&= \Tr_0  K^-_0(\pm r^{-1})^t \bigl(R_{0N}(\pm r z_N^{-1})^{-1}\bigr)^{t_0} \cdots \bigl( R_{01}(\pm r z_1^{-1})^{-1} \bigr)^{t_0} M_0^t \cdot \\
& \hspace{20mm} \cdot  R_{01}(\pm r^{-1} z_1^{-1})^{t_0} \cdots R_{0N}(\pm r^{-1} z_N^{-1})^{t_0}.
\end{align*}
Combining \eqrefs{CS}{RMM} yields, for generic values of $\bm z$,
\[ M_0^t \bigl( R_{0i}(\pm  r^{-1} z_i^{-1} )^{t_0} \bigr)^{-1} \propto \bigl( R_{0i}(\pm  r z_i^{-1})^{-1} \bigr)^{t_0} M_0^t. \]
Repeatedly applying this we obtain $\ca T(\pm r^{-1};\bm z) \propto \Tr_0 K^-_0(\pm r^{-1})^{t} M_0^t$, from which the Lemma follows after applying \eqref{traceoftransposes}.
\end{proof}

\end{document}